\documentclass[12pt, reqno]{amsart}
\usepackage{amsmath, amsthm, amscd, amsfonts, amssymb, graphicx, color}
\usepackage[bookmarksnumbered, colorlinks, plainpages]{hyperref}
\input{mathrsfs.sty}

\hypersetup{colorlinks=true,linkcolor=red, anchorcolor=green,
citecolor=cyan, urlcolor=red, filecolor=magenta, pdftoolbar=true}

\newtheorem{twr}{Theorem}[section]
\newtheorem{lem}[twr]{Lemma}

\theoremstyle{definition}

\newtheorem{example}[twr]{Example}

\newtheorem{obs}[twr]{Observation}
\newtheorem{remark}[twr]{Remark}
\numberwithin{equation}{section}

\DeclareMathOperator{\conv}{conv}
\DeclareMathOperator{\absconv}{absconv}

\DeclareMathOperator{\lin}{lin}

\DeclareMathOperator{\tr}{Tr}

\DeclareMathOperator{\ext}{ext}

\DeclareMathOperator{\re}{Re}
\DeclareMathOperator{\im}{Im}
\usepackage{color}
\usepackage{enumerate}

\begin{document}

\setcounter{page}{1}

\title[$2$-strong uniqueness of a best approximation and minimal projections]{$2$-strong uniqueness of a best approximation and of minimal projections in complex polytope norms and their duals}

\author[T. Kobos \MakeLowercase{and} G. Lewicki]
{Tomasz Kobos$^1$ \MakeLowercase{and} Grzegorz Lewicki$^2$}
\address{$^1$Faculty of Mathematics and Computer Science, Jagiellonian University, \L ojasiewicza 6, 30-348 Krak\'ow, Poland}
\email{tomasz.kobos@uj.edu.pl}
\address{$^2$Faculty of Mathematics and Computer Science, Jagiellonian University, \L ojasiewicza 6, 30-348 Krak\'ow, Poland}
\email{grzegorz.lewicki@uj.edu.pl}

\subjclass[2010]{47A58, 41A65, 52A21.}
\keywords{}
\begin{abstract}
We study a property of $2$-strong uniqueness of a best approximation in a class of finite-dimensional complex normed spaces, for which the unit ball is an absolutely convex hull of finite number of points and in its dual class. We prove that, contrary to the real case, these two classes do not coincide but are in fact disjoint. We provide several examples of situations in these two classes, where a uniqueness of an element of a best approximation in a given subspace implies its $2$-strong uniqueness. In particular, such a property holds for approximation in an arbitrary subspace of the complex $\ell_1^n$ space, but not of the complex $\ell_{\infty}^n$ space. However, this is true in general under an additional assumption that a subspace has a real basis and an ambient complex normed space is generated by real vectors or functionals. We apply our results and related methods to establish some results concerned with $2$-strongly unique minimal projections in complex normed spaces, proving among other things, that a minimal projection onto a two-dimensional subspace of an arbitrary three-dimensional complex normed space is $2$-strongly unique, if its norm is greater than $1$.
\end{abstract} \maketitle

\section{Introduction}
Let $X$ be a normed space over $\mathbb{K}$ (where $\mathbb{K}=\mathbb{R}$ or $\mathbb{K}=\mathbb{C}$) and let $Y \subseteq X$ be a linear subspace. We say that $y_0 \in Y$ is a \emph{best approximation} in $Y$ for a vector $x \in X \setminus Y$ if 
$$\|x-y\| \geq \|x-y_0\|$$
for every $y \in Y$. Finding a best approximation of a given element and determining whether it is unique or not, is a fundamental problem of approximation theory. Even when the best approximation is determined to be unique, it is often highly desirable (also from the numerical point of view) to establish some additional properties of an element of a best approximation. A widely known example of such a property is the strong uniqueness. We say that $y_0 \in Y$ is a \emph{strongly unique best approximation} for $x \in X \setminus Y$ if there exists a constant $r>0$ such that
$$\|x-y\| \geq \|x-y_0\| + r\|y-y_0\|$$
for every $y \in Y$. A strong uniqueness of $y_0$ provides much more information than just a uniqueness, as it gives an explicit lower bound on the distance of $x$ to an arbitrary element of $y$, quantifying how $y$ is "non-optimal" for $x$. It is not hard to see, that a strong uniqueness of a best approximation for $x$ yields a local Lipschitz continuity at $x$ of the metric projection onto $Y$. At the same time, it is known, that in a completely general case, the metric projection can even be discontinuous. This can happen in the infinite-dimensional setting, even when $Y$ is assumed to be a \emph{Chebyshev} subspace (i.e. every vector $x \in X$ has a unique best approximation in $Y$, so that a metric projection is a single valued mapping).

The property of a strong uniqueness of a best approximation originates from the work of Newman and Shapiro from $1963$ \cite{newmanshapiro}, who proved that if $Y$ is a finite-dimensional Haar subspace of the space of real continuous functions $C(K)$ (where $K$ is a compact Hausdorff space), then every function $f \in C(K)$ has a strongly unique best approximation in $Y$. This was a significant improvement of a much older and famous result, obtained independently around the beginning of the 20th century by Borel, Kircherberger and Young, that every finite-dimensional Haar subspace of $C(K)$ is a Chebyshev subspace (the converse statement was established later by Haar). After that, the property of strong uniqueness and its relatives were widely studied by many authors in a variety of different contexts and its importance is well established. Here we shall provide only some necessary background, but for a thorough introduction to the topic of strong uniqueness we refer the reader to a survey paper \cite{survey}, which contains a great overview of the existing results, along with some applications of this notion. In this survey paper one can find references to more than one hundred papers related to this concept. It is worth mentioning, that already in $1943$, i.e. well before the work of Newman and Shapiro, Chebotarev obtained some results in \cite{chebotarev} that were very closely related to the concept of strong uniqueness (see \cite{marinov} for a discussion of the contribution of Chebotarev). 

In the current paper we will be interested in a more general notion of an $\alpha$-strong uniqueness. If $X$ is a normed space, $Y \subseteq X$ is a subspace and $x \in X \setminus Y$, then $y_0$ is called an \emph{$\alpha$-strongly} unique best approximation for $x$ in $Y$ (where $\alpha \geq 1$ is a real parameter), if there exists a constant $r>0$ such that
$$\|x-y\|^{\alpha} \geq \|x-y_0\|^{\alpha} + r\|y-y_0\|^{\alpha}.$$
for every $y \in Y$. Thus, $1$-strong uniqueness coincides with the strong uniqueness defined before. It is not hard to see that $\alpha$-strong uniqueness implies $\alpha'$-strong uniqueness for any $\alpha' \geq \alpha$ with the same constant $r$. In other words, for smaller $\alpha$ we get a stronger property. Similarly like $1$-strong uniqueness, the property of $\alpha$-strong uniqueness implies a local H\"older continuity of the metric projection onto $Y$ at $x$.

The $1$-strong uniqueness is a property that is most easily found in spaces like $C(K)$ or $\ell_1$, where the set of extreme points of the unit ball is, roughly speaking, not too large. For example, in the space $\mathbb{R}^n$ with the standard Euclidean norm, every closed subspace is a Chebyshev subspace, but there is no $1$-strong uniqueness. In this case, it is easy to prove that the best approximation in a subspace is always $2$-strongly unique, but never $\alpha$-strongly unique for $\alpha<2$.

In this paper we shall be concerned with the finite-dimensional setting. In the real finite-dimensional case, the norms with a finite set of extremal points of the unit ball form a broad class of norm with a prevalence of strong uniqueness. Such norms in $\mathbb{R}^n$ will be called \emph{real polytope norms} and the corresponding normed spaces will be called \emph{real polytope spaces}. Most obvious examples include spaces $\ell_1^n$ or $\ell_{\infty}^n$, but any $0$-symmetric convex polytope with a non-empty interior in $\mathbb{R}^n$ corresponds to such a norm (and thus such norms are dense in $\mathbb{R}^n$). It should be noted that traditionally such norms are often called \emph{polyhedral norms}, but we will use the terminology of real polytope norms for the sake of consistency with the complex case, that shall be discussed shortly. A remarkable property of real polytope norms is given in the following result, which was most likely observed first by Jittorntrum and Osborne in \cite{jitto} (see also \cite{sudolski} and \cite{li} for closely related results).

\begin{twr}
\label{twrsudolskiwojcik}
Let $X = (\mathbb{R}^n, \| \cdot \|)$ be a real polytope normed space. If $Y \subseteq X$ is a subspace and $x \in X \setminus Y$ is a vector with a unique best approximation $y_0 \in Y$, then $y_0$ is a $1$-strongly unique best approximation for $x$.
\end{twr}

Therefore, in real polytope spaces a uniqueness of best approximation automatically implies a $1$-strong uniqueness (although in these norms there can be no uniqueness of best approximation at all). It seems however, that there is no straightforward way to transfer this result to the complex setting, where the situation turns out to be vastly different.

Our main goal is to develop some properties similar to that in Theorem \ref{twrsudolskiwojcik} in a broad class of finite-dimensional complex normed spaces and provide some applications of these results to norm-minimal projections. It must be noted that problems we consider are profoundly different in the complex setting. To give some informal explanation of this fact, let us note that the modulus of a complex number $|a+bi|=\sqrt{a^2+b^2}$ is a Euclidean norm of a vector $(a, b) \in \mathbb{R}^2$ and for this reason, complex norms often behave akin to the Euclidean norm, where there is a $2$-strong uniqueness, but no an $\alpha$-strong uniqueness for any $\alpha<2$. It is therefore not surprising that a $1$-strong uniqueness is rarely found in complex normed spaces and a $2$-strong uniqueness often plays it it role instead. This is illustrated by a complex version of the previously mentioned result of Newman and Shapiro \cite{newmanshapiro}, which still holds, but with a $1$-strong uniqueness replaced by $2$-strong uniqueness. This was proved already in the paper of Newman and Shapiro, but for reference see also Section $10$ in \cite{survey} or \cite{smarzewski}. It should be noted however, that some non-trivial examples of a $1$-strong uniqueness occurring in complex spaces like $C(K)$ or $L_1$ do exist (see for instance Theorem $10.6$ in \cite{survey} and examples given afterwards).

When trying to transfer Theorem \ref{twrsudolskiwojcik} to the complex case it is therefore natural to replace a $1$-strong uniqueness with a $2$-strong uniqueness, that is generally more suitable for the complex scalars. However, there is also another important point of contrast with the real setting that should be taken into consideration. It is a basic fact of convex geometry that polytope norms in $\mathbb{R}^n$ form a self-dual class. In other words, every polytope norm $\| \cdot \|$ in $\mathbb{R}^n$ can be described for $x \in \mathbb{R}^n$ as
$$\|x\| = \max_{1 \leq j \leq N} |f_j(x)|$$
for some functionals $f_1, \ldots, f_N: \mathbb{R}^N \to \mathbb{R}$ (i.e. the unit ball of the dual norm is equal to $\conv \{\pm f_1, \ldots, \pm f_N \}$). On the other hand, the situation in the complex case is completely different. Adhering to a terminology from \cite{complexpolytope} we shall call a norm $\| \cdot \|$ in $\mathbb{C}^n$ a \emph{complex polytope norm} if there are vectors $u_1, \ldots, u_N \in \mathbb{C}^n$ such that for the unit ball $B$ of $\| \cdot \|$ we have
$$B = \absconv \{u_1, \ldots, u_N\}$$
$$=\left \{ \alpha_1u_1 + \ldots +  \alpha_Nu_N: \ \alpha_1, \ldots, \alpha_N \in \mathbb{C}, \ |\alpha_1| + \ldots + |\alpha_N| \leq 1 \right \}.$$
The corresponding complex normed space space will be called a \emph{complex polytope space}. A minimal set of vectors $u_j$'s as above will be called an \emph{essential system of vertices} (it is defined uniquely up to permutation and multiplication by scalars of modulus $1$). A general set of the form $\absconv\{u_1, \ldots, u_N\}$ will be called a \emph{balanced complex polytope}. The dual norm to a complex polytope norm will be called an \emph{adjoint complex polytope norm} and the resulting normed space an \emph{adjoint complex polytope space}. If $\| \cdot \|$ is an adjoint complex polytope norm, then for every $x \in \mathbb{C}^n$ we have
$$\|x\| = \max_{1 \leq j \leq N} |f_j(x)|$$
for some functionals $f_1, \ldots, f_N: \mathbb{C}^N \to \mathbb{C}$. A minimal set of such functionals will be called an \emph{essential system of facets} of the norm $\| \cdot \|$. Clearly, in this case, the dual norm is a complex polytope norm with an essential system of vertices $f_1, \ldots, f_N$, when functionals $f_j$ are treated as vectors in $\mathbb{C}^n$. 

A basic example of a complex polytope norm in $\mathbb{C}^n$ is the $\ell_1$ norm
$$\|x\|_1 = |x_1| + \ldots + |x_n|,$$
with an essential system of vertices being the vectors from the canonical unit basis. A basic example of an adjoint complex polytope norm in $\mathbb{C}^n$ is the $\ell_{\infty}$ norm
$$\|x\|_{\infty} = \max_{1 \leq j \leq n} |x_j|,$$
with an essential system of facets again corresponding to the canonical unit basis. 

Contrary to the real case, the class of $n$-dimensional complex polytope spaces does not coincide with the class of $n$-dimensional adjoint complex polytope spaces, but they are actually disjoint, so that a norm can not have both properties at the same time (with an exception of dimension $1$). This fact was explicitly stated in \cite{complexpolytope}, although without a proof. We fill out this gap, by providing one in Section \ref{sectionnorms} (see Theorem \ref{twrnorms}). Expectedly, both classes are dense in the class of all $n$-dimensional complex normed spaces (see Theorems $5.7$ and $5.8$ in \cite{complexpolytope}). For a further background on the complex polytope spaces and the adjoint complex polytope spaces, the reader is referred to \cite{complexpolytope}, where properties of both classes are systematically developed, with some examples of contrast with the real setting being highlighted. A thorough study of the two-dimensional balanced complex polytopes can be found in \cite{complexpolytope2dim}.

From the viewpoint of approximation theoretic properties investigated in this paper, it seems that there is no general reason to believe that the properties of complex polytope and adjoint complex polytope spaces should be somehow the same. Therefore, it should not be expected that the same properties should hold for both class and it is clear from the practical point of view, that these two settings usually require different approaches. One crucial difference between these two classes is the fact that for an adjoint complex polytope norm we have a simple formula for computing the norm
$$\|x\| = \max_{1 \leq j \leq N} |f_j(x)|,$$
for some functionals $f_1, \ldots, f_N: \mathbb{C}^N \to \mathbb{C}$, which is very useful when working with some problems involving the distances. However, if $\| \cdot \|$ is a comlex polytope norm with an essential system of vertices $u_1, \ldots, u_N$ then $\| \cdot \|$ can be determined as
$$\|x\| = \min\{|a_1| + \ldots + |a_N|: \  a_1, \ldots, a_N \in \mathbb{C} \ \text{ and } \  x = a_1u_1 + \ldots a_Nu_N\}.$$
It is rather clear that this formula is much less convenient than the one for an adjoint complex polytope norm -- when $N$ is large, a vector $x \in \mathbb{C}^n$ has many representations as a linear combination of $u_1, \ldots, u_N$ and it is a highly non-trivial problem in general to determine the representation achieving the minimum (even in the two-dimensional case, as demonstrated in \cite{complexpolytope2dim}). Such a difficulty obviously does not exist in the case of $N=n$. In this situation the norm is easily seen to be isometric to the $\ell_1$ norm and the representation is unique, yielding a simple formula.

The paper is organized as follows. In Section \ref{sectionnorms} we give a proof of the fact that for $n \geq 2$ a complex polytope space in $\mathbb{C}^n$ is never an adjoint complex polytope space (Theorem \ref{twrnorms}). In Section \ref{sectionproperties} we provide some basic properties of an element best approximation in complex polytope spaces and their duals (see Lemmas \ref{charadjoint} and \ref{lemchar}). It turns out that a property like in Theorem \ref{twrsudolskiwojcik} holds with $2$-strong uniqueness for $1$-dimensional subspaces of every adjoint complex polytope space (Observation \ref{obs1dim}). However, it fails already for dimension $2$ and a counterexample can be found in the space $\ell_{\infty}^4$ (Example \ref{counterexample}). In the setting of complex polytope spaces the situation is much less clear, but certainly a counterexample to such property can not be found in the space $\ell_1^n$, in which a uniqueness automatically implies a $2$-strong uniqueness (Theorem \ref{twrl1}). In Section \ref{sectionreal} we continue this theme, presenting how to establish a property like in Theorem \ref{twrsudolskiwojcik} for complex polytope spaces and their duals, but under some additional assumptions. When a complex polytope norm has an essential system of vertices consisting of real vectors or when an adjoint complex polytope norm has an essential system of facets corresponding to real functionals, then uniqueness automatically implies $2$-strong uniqueness for subspaces generated by real vectors. This is proved in Theorem \ref{twrpolytope}, which also provides a notable difference between complex polytope and adjoint complex polytope settings. Namely, the order of $2$ of $2$-strong uniqueness can never be improved for adjoint complex polytope spaces, but in the complex space $\ell_1^n$ it can be improved to $1$-strong uniqueness (Theorem \ref{twrl12}).

In Section \ref{sectionproj} we apply results of the previous sections and other methods originating from approximation theory, to establish some new results concerned with a $2$-strong uniqueness of norm-minimal projections in normed spaces. Problems related to minimal projections and projection constants of normed spaces are classical problems of the geometry of Banach spaces and they were studied extensively for many years, often bringing significant challenges. Let us recall that if $X$ is normed space and $Y \subseteq X$ is a subspace, then the \emph{relative projection constant} of $Y$ is defined as $\lambda(Y, X) = \inf  \|P\|$, where the infimum is taken over all linear and continuous projections $P: X \to Y$ (i.e. $P(y)=y$ for $y \in Y$) and the standard operator norm is considered. A projection $P$ satisfying $\|P\|=\lambda(Y, X)$ will be called a \emph{minimal projection} from $X$ onto $Y$. In a completely general case, a minimal projection does not have to exist (there even may be no continuous projections at all), but in the finite-dimensional setting that we shall consider, its existence follows easily from a standard compactness argument. Clearly we always have $\lambda(Y, X) \geq 1$ and if the equality $\lambda(Y, X)=1$ holds, we say that $Y$ is a \emph{$1$-complemented} subspace of $X$. It turns out, that determining a minimal projection in practice is often a very challenging task. Even when a given minimal projection $P$ is already established to be minimal, it is usually still non-trivial to determine if it is a unique minimal projection or not. Similarly to a uniqueness of an element of best approximation, a stronger property of uniqueness of a minimal projection can be defined. We say that a projection $P: X \to Y$ is an \emph{$\alpha$-strongly unique minimal projection} (where $\alpha \geq 1$) if there exists a constant $r>0$ such that
$$\|Q\|^{\alpha} \geq \|P\|^{\alpha} + r\|Q-P\|^{\alpha},$$
for any linear and continuous projection $Q: X \to Y$. Using the fact that the problem of determining whether given projection is minimal or not can be equivalently stated as a problem of a best approximation in a space of linear operators, we are able to apply methods of approximation theory, to establish $2$-strong uniqueness of some minimal projections in complex normed spaces. We note that up to this time, a property of strong uniqueness of minimal projections was mainly studied in the most classical case of $1$-strong uniqueness (see for example \cite{baronti}, \cite{lewicki}, \cite{lewickimicek}, \cite{martinov}, \cite{martinov2}, \cite{odyniecprophet}). In particular, we are able to extend a result of Odyniec from $1978$ to the complex setting, proving that for non $1$-complemented two-dimensional subspaces of three dimensional complex normed spaces a minimal projection is not only unique, but even $2$-strongly unique (Theorem \ref{twrcompl}). To our best knowledge, even the uniqueness was not established before in the complex case. Moreover, while a minimal projection onto a hyperplane of the complex $\ell_{\infty}^n$ and $\ell_1^n$ spaces does not have to be unique, if it is unique, then it is automatically $2$-strongly unique and, at least in the case of $\ell_{\infty}^n$, the order can not be improved for non $1$-complemented hyperplanes (Theorems \ref{twrlinfty} and \ref{twrl1proj}). In Theorem \ref{twrprojreal} we obtain a much more general result, which is related to result of Section \ref{sectionreal} -- uniqueness of a minimal projection onto real subspaces of complex polytope spaces and adjoint complex polytope spaces generated by real vectors or real functionals always implies $2$-strong uniqueness. 

Throughout the paper by $B_X$ and $S_X$ we shall always denote the unit ball and unit sphere of a normed space $X$ over $\mathbb{K}$, respectively. The set of \emph{extreme} points of a set $C \subseteq \mathbb{K}^n$ will be understood in a usual way, that is a point $a \in C$ is an extreme point of $C$, if there do not exist $x, y \in C$ and $t \in (0, 1)$ such that $a = tx + (1-t)y$. An \emph{absolutely convex hull} of a set $C \subseteq \mathbb{K}^n$ is defined in a similar way for real and complex scalars as
$$\absconv C = \left \{ \alpha_1x_1 + \ldots +  \alpha_Nx_N: \ x_1, \ldots, x_N \in C, \ \alpha_1, \ldots, \alpha_N \in \mathbb{K}, \ \sum_{j=1}^{n} |\alpha_j| \leq 1 \right \}. $$
To avoid some trivial considerations, we will consider only normed spaces of dimension at least two and by a term \emph{subspace} we shall always mean a proper linear subspace (with some non-zero element but not equal to the whole space).

\section{Complex polytope spaces and their duals}
\label{sectionnorms}

A fundamental theorem of convex geometry states that a convex polytope in $\mathbb{R}^n$, i.e. convex hull of a finite number of points, can be equivalently represented as an intersection of finitely many half spaces. In the case of the origin-symmetric polytopes, this yield immediately that if the unit ball of a normed space $X=(\mathbb{R}^n, \| \cdot \|)$ is a polytope, then the same is true for the unit ball of the dual norm. The aim of this short section is to provide a proof of fact that in the complex setting the situation is completely opposite, that is, a complex polytope norm is never a dual of a complex polytope norm in $\mathbb{C}^n$ (where $n \geq 2$). While this lies perhaps outside of the main scope of the paper, we believe that this fact represents a truly remarkable contrast between real and complex polytopes and this matter requires some clarification. This was property was stated as an obvious fact in \cite{complexpolytope} (see for example after equation $(52)$ or after equation $(69)$), although no formal explanation was provided there. After a closer inspection, it seems that this property does not follow in any direct way from the results presented in this paper, contrary to what authors believed at that time (confirmed in private communication). We emphasize that we are not aware of any other mention of this result in the existing literature.

It is easy to observe that an adjoint complex polytope space is isometric to a subspace of the complex $\ell_{\infty}^M$ space, while a complex polytope space is isometric to a quotient of the complex $\ell_{1}^N$. Thus, an equivalent statement is that a subspace of the complex $\ell_{\infty}^M$ is never isometric to a quotient of the complex $\ell_{1}^N$, with an obvious exception of $1$-dimensional subspaces. This is clearly not true for the real scalars, but this is also not true if we would consider the complex $\ell_{\infty}$ and $\ell_1$ spaces instead, as it is well-known that every seperable Banach space (real or complex) can be embedded isometrically into $\ell_{\infty}$ and is isometric to a quotient of $\ell_1$. Yet another reformulation, this time using the language of the real geometry, is provided by embedding everything into $\mathbb{R}^{2n}$ in a standard way. In this case, it is easy to see that the unit ball of a complex polytope norm is a convex hull of finitely many Euclidean discs in $\mathbb{R}^{2n}$, while the unit ball of an adjoint polytope norm is an intersection of finitely many cylinders, each having some Euclidean disc as a base. It turns out that, with the exception of $n=1$ (when both convex bodies reduce to just a single Euclidean disc in $\mathbb{R}^2$), such convex bodies never coincide.

Our approach is based on proving the existence of a face with exactly two vertices of any balanced complex polytope $\absconv \{u_1, \ldots, u_N\} \subseteq \mathbb{C}^n$. Then it is relatively easy to conclude that the unit ball of the dual space can not be generated by a finite numbers of functionals. We start with a simple lemma about complex numbers.

\begin{lem}
\label{lemcomplex}
For all pairwise distinct complex numbers $x_1, x_2, \ldots, x_m \in \mathbb{C}$ (where $m \geq 2$) there exists $t \in \mathbb{C}$ with $|t|>0$ arbitrary small such that for some $1 \leq k < l \leq m$ we have
$$|tx_k+1|=|tx_l+1| > |tx_j+1|$$
for any $j \neq k, l$.
\end{lem}

\begin{proof}

Let us consider a convex polygon (possibly a segment), which is a convex hull of numbers $x_1, \ldots, x_m$, when considered as points on the complex plane. Without loss of generality we can assume that $x_1, x_2$ are consecutive vertices of this polygon. We shall prove that such complex number $t$, as in the statement, exists for $x_1, x_2$. The desired condition can be written as
$$\left | \frac{-1}{t} - x_1  \right | = \left | \frac{-1}{t} - x_2  \right |  > \left | \frac{-1}{t} - x_j  \right |$$
for any $j>2$. Therefore, thinking geometrically, we seek a point equidistant to $x_1, x_2$ with the distance larger than to any other point $x_j$ for $j>2$. Since $x_1x_2$ is a side of a convex hull of $x_1, x_2, \ldots, x_m$, all of the points lie on one side of the line $x_1x_2$. It is now evident that every point lying on the perpendicular bisector of the line $x_1x_2$, which is sufficiently far away and lying on the same side of this line as all of the points, satisfy the desired condition. Since the number $|\frac{1}{t}|$ can be arbitrarily large, the number $|t|$ can be arbitrarily small and the conclusion follows. 

\end{proof}

In the following lemma, which is crucial for the proof of Theorem \ref{twrnorms}, we establish the existence of a linear functional which is maximized on exactly two vertices of a balanced complex polytope. This result can be considered intuitively clear, as this corresponds just to the existence of an edge of a polytope in a real case. However, we emphasize that a lot of caution is needed when transferring the intuition from the real case to the complex setting. In the real case an edge of a polytope could be also simply defined as a $1$-dimensional face. However, in the complex case it is no longer true that $1$-dimensional faces have necessarily only two vertices -- see for instance Example $3.17$ in \cite{complexpolytope} or Section $3$ in \cite{complexpolytope2dim}, which contains a long and detailed analysis of the case of balanced complex polytopes in $\mathbb{C}^2$. In particular, when a balanced complex polytope in $\mathbb{C}^2$ has exactly three essential vertices, there always exist two $1$-dimensional faces with exactly three vertices (Theorem $3.1$). Such a face is called there \emph{special}, while a face with exactly $2$ vertices is called \emph{regular}. Nevertheless, the analysis provided in this paper shows in particular the existence of regular faces for every number of vertices of a balanced complex polytope in $\mathbb{C}^2$. We provide a direct proof of this fact, which works for every dimension $n \geq 2$.

\begin{lem}
\label{lemface}
Let $n \geq 2$ and let $u_1, \ldots, u_N \in \mathbb{C}^n$ be vectors spanning $\mathbb{C}^n$ such that $u_k \not \in \absconv\{u_j: \ j \neq k\}$ for every $1 \leq k \leq N$. Then there exist $1 \leq k < l \leq N$ and a linear functional $f: \mathbb{C}^n \to \mathbb{C}$ satisfying
$$|f(u_k)|=|f(u_l)|=1$$
and $|f(u_j)|<1$ for every $j \neq k, l$.
\end{lem}

\begin{proof}
Let $X = (\mathbb{C}^n, \| \cdot \|)$ with $B_X = \absconv \{u_1, \ldots, u_N\}$. We start by observing that there exists a functional $g: \mathbb{C}^n \to \mathbb{C}$ such that $|g(u_k)| \neq |g(u_l)|$ for any $1 \leq k < l \leq N$. Indeed, using the standard inner-product in $\mathbb{C}^n$, this can be restated as $|\langle x, u_k \rangle | \neq |\langle x, u_l \rangle |$ for some $x \in \mathbb{C}^n$ and all $1 \leq k < l \leq N$. Assuming otherwise, we see that the sets 
$$V_{k, l} = \{ x \in \mathbb{C}^n: |\langle x, u_k \rangle | = \langle x, u_l \rangle|  \}$$
would cover the space $\mathbb{C}^n$, i.e. $\bigcup_{1 \leq k < l \leq N} V_{k, l} = \mathbb{C}^n$. As these sets are obviously closed, it follows that at least one of them has non-empty interior, as $\mathbb{C}^n$ is not a finite union of nowhere dense sets. Suppose that $rB_X + v \subseteq V_{k, l}$ for some $v \in \mathbb{C}^n$ and $r>0$. In other words, for every vector $x \in \mathbb{C}^n$ with $\|x\| \leq r$ we have
$$|\langle v+x, u_k \rangle | = |\langle v+x, u_l \rangle |.$$
For $x=0$ we obtain in particular that $| \langle v, u_k \rangle | = | \langle v, u_l \rangle |$. Now, let $x \in \mathbb{C}^n$, $\|x\| \leq r$ be a vector that is perpendicular to $u_l$ (i.e. $\langle x, u_l \rangle = 0$) but not to $u_k$ (such vector exists, as by the assumption, $u_k$ and $u_l$ are not linearly dependent). Furthermore, we multiply $x$ by a suitable scalar of modulus $1$, so that $\frac{\langle v, u_k \rangle}{\langle x, u_k \rangle}$ is a non-negative real number. In this case we have
$$ |\langle v+x, u_k \rangle | = |\langle v, u_k \rangle| + |\langle x, u_k \rangle| > |\langle v, u_k \rangle | = | \langle v, u_l \rangle| = |\langle v + x, u_l \rangle |.$$
We obtained a contradiction and in consequence, there exists a linear functional $g: \mathbb{C}^n \to \mathbb{C}^n$ satisfying $|g(u_k)| \neq |g(u_l)|$ for every $1 \leq k < l \leq N$.

On the other end of the spectrum, it is hard to see that there exists a functional $h \in S_{X^*}$ such that $|h(u_j)|=1$ for at least two different indices $1 \leq j \leq N$. Indeed, obviously there exists a point on the unit sphere $S_X$, which is not of the form $tu_j$ for $|t|=1$ and $1 \leq j \leq N$. Let us suppose that it is of the form $x=a_1u_1 + \ldots + a_mu_m$, where $|a_1| + \ldots + |a_m|=\|x\|=1$, $a_j \neq 0$ and $m \geq 2$. Now if $h \in S_{X^*}$ is a functional such that $h(x)=1$, then from the triangle inequality it easily follows that $|h(u_j)|=1$ for any $1 \leq j \leq m$. Thus, let us further assume that the equality $|h(u_j)|=1$ holds exactly for the indices $1 \leq j \leq m$ and for $m+1 \leq j \leq N$ we have $|h(u_j)|<1$ (possibly $m=N$). 

We shall find a certain combination of the functional $g$ (which differentiates all vectors $u_j$) and the functional $h$ (which maximizes on at least two of them), so that the resulting functional will satisfy the desired condition. To this end, let us define $R_1<1$ as the maximum of $|h(u_j)|$ with $m+1 \leq j \leq N$ (or we take $R_1=0$ if $m=N$). By multiplying $u_j$ for $1 \leq j \leq m$ by an appropriate scalar of modulus $1$ we can assume that $h(u_j)=1$ for every such $j$. From the assumption about $g$, it follows that, even after such a readjustment, the complex numbers $g(u_1), \ldots, g(u_m)$ are pairwise distinct. Hence, by Lemma \ref{lemcomplex} we can find $1 \leq k, l \leq m$ and $t \in \mathbb{C}$ with $|t|>0$ arbitrary small such that
$$|tg(u_k)+1| = |tg(u_l)+1| > |tg(u_j)+1|,$$
for $1 \leq j \leq m$, $j \neq k, l$. Now, let us define a functional $f: \mathbb{C}^n \to \mathbb{C}$ as
$$f(x)=tg(x)+h(x)$$
for a given $t$ satisfying the previous condition. Clearly, we have
$$|f(u_k)|=|f(u_l)|>|f(u_j)|$$
for $1 \leq j \leq m$, $j \neq k, l$. Moreover, if $m+1 \leq j \leq N$, then
$$|f(u_j)| = |tg(u_j)+h(u_j)| \leq |t| |g(u_j)| + |h(u_j)| \leq |t|R_2 + R_1,$$
where $R_2$ is a maximum of $|g(u_j)|$ with $m+1 \leq j \leq N$ (or again $R_2=0$ for $m=N$). On the other hand
$$|f(u_k)| = |tg(u_k)+1| \geq 1 - |t||g(u_k)|$$
and similarly for $u_l$. Let $R_3 = \max\{|g(u_k), |g(u_l)|\}$. Thus, since $R_1<1$, we see that the inequality
$$|t|R_2 + R_1 < 1 - |t|R_3$$
is satisfied for $|t|$ sufficiently small, i.e. for $|t| < \frac{1-R_1}{R_2+R_3}$. Thus, after a proper rescaling the functional $f$ satisfies the desired condition and conclusion follows.
\end{proof}

With a functional from the previous lemma, it is not hard to construct infinitely many linear functionals of norm $1$, which can not be contained in absolutely convex hull of finitely many functionals. 

\begin{twr}
\label{twrnorms}
There does not exist a norm in $\mathbb{C}^n$ (where $n \geq 2$), which is simultaneously a complex polytope norm and an adjoint complex polytope norm.
\end{twr}

\begin{proof}

Let us suppose that $X = (\mathbb{C}^n, \| \cdot \|)$ is a complex normed space with a norm such that $B_X=\absconv\{u_1, \ldots, u_N\}$ for some vectors $u_1, \ldots, u_N \in \mathbb{C}^n$ that span $\mathbb{C}^n$ and satisfy $u_k \not \in \absconv\{u_j: \ j \neq k\}$ for every $1 \leq k \leq N$. Moreover, we assume also that $B_{X^*} = \absconv \{f_1, \ldots, f_M\}$ for some functionals $f_1, \ldots, f_M: \mathbb{C}^n \to \mathbb{C}$, aiming for a contradiction. From the previous lemma we know that there exist $1 \leq k < l \leq N$ and a linear functional $f: \mathbb{C}^n \to \mathbb{C}$ satisfying
$$|f(u_k)|=|f(u_l)|=1$$
and $|f(u_j)|<1$ for every $j \neq k, l$. By multiplying $u_k$, $u_l$ by appropriate scalars of modulus $1$ we can suppose additionally that $f(u_k)=f(u_l)=1$. Clearly, since $u_k$ and $u_l$ are linearly independent, for any scalar $t \in \mathbb{C}$ of modulus $1$ that is sufficiently close to $1$ on the complex unit circle, we can find a functional $f^t: \mathbb{C}^n \to \mathbb{C}$ satisfying $f^t(u_k)=1$, $f^t(u_l)=t$ and $|f^t(u_j)|<1$ for every $j \neq k, l$. Obviously we have $\|f_t\|=1$ for any such $t$. Now, since $f^t \in  \absconv \{f_1, \ldots, f_M\}$ by the assumption, for any fixed $t$ we can write
$$f^t=a_1f_1 + \ldots + a_sf_s,$$
where $|a_1| + \ldots + |a_s| \leq 1$ and we can suppose additionally that $a_j \neq 0$ for $1 \leq j \leq s$ (coefficients $a_j$ depend on $t$). In this case
$$1=f^t(u_k) = a_1f_1(u_k) + \ldots + a_sf_s(u_k).$$
Since $|f_j(u_k)| \leq 1$ for every $1 \leq j \leq s$, it follows from the triangle inequality that
$$|a_1f_1(u_k) + \ldots + a_sf_s(u_k)| \leq |a_1||f_1(u_k)| + \ldots + |a_s||f_s(u_k)| \leq |a_1| + \ldots + |a_s| \leq 1.$$
Since the equality holds, we must have $|f_j(u_k)|=1$ and $a_jf_j(u_k)=|a_j|$ for every $1 \leq j \leq s$. Similarly we establish that $|f_j(u_l)|=1$ and $a_jf_j(u_l)=t|a_j|$ for $1 \leq j \leq s$. In particular, for $j=1$ we have
$$t = \frac{f_1(u_l)}{f_1(u_k)}.$$
However, this is a contradiction, since $t$ can take infinitely many values (that are sufficiently close to $1$ on the complex unit disc) and the set
$$ \left \{\frac{f_j(u_l)}{f_j(u_k)}: 1 \leq j \leq M \quad \text{ and } \quad |f_j(u_l)|=|f_j(u_k)|=1 \right \}.$$
is obviously finite. Hence the norm $\| \cdot \|$ is not an adjoint complex polytope norm and the proof is finished.
\end{proof}

\section{Properties of a best approximation in complex polytope spaces and their duals}
\label{sectionproperties}

In this section we turn our attention to the problems concerned with approximation in subspaces. We start with recalling the following basic characterization of an element of a best approximation in terms of linear functionals, which works for all normed spaces (for a proof see for example Theorem $1.1$ in \cite{singer}). Throughout the paper we will often assume for simplicity that $0$ is a best approximation for $x$ in a subspace $Y$. However, this assumption is usually not restrictive at all, since if $y_0$ is a best approximation for $x$ in $Y$, then the vector $x-y_0$ has $0$ as a best approximation.

\begin{lem}
\label{lemfunkcjonal}
Let $X$ be a finite-dimensional normed space over $\mathbb{K}$ and let $Y \subseteq X$ be a subspace. If $x \in X \setminus Y$, then $0 \in Y$ is a best approximation for $x$ in $Y$ if and only if, there exists a functional $f \in S_{X^*}$ such that $f(x)=\|x\|$ and $f|_{Y} \equiv 0$. 
\end{lem}

In a more specific case of the complex adjoint polytope norms, which are defined by a finite set of functionals, the lemma above readily implies a convenient characterization of the element of best approximation and its uniqueness. This characterization can not be really considered to be new, since an adjoint complex polytope norm $\|x\|=\max_{1 \leq j \leq N} |f_j(x)|$ can be isometrically embedded into a complex $\ell_{\infty}^N$ space through the mapping $x \to (f_1(x), \ldots, f_N(x))$. One can then regard $\ell_{\infty}^N$ as the space $C(K)$ for $K=\{1, \ldots, N\}$ and refer to a similar well-known characterization for the space $C(K)$, largely retrieving the lemma below (see for example Corollary $10.2$ in \cite{survey}). Nevertheless, for a purpose of a clarity and completeness of the exposition, it is prudent to have the result stated and proved specifically for the adjoint complex polytope norms.

\begin{lem}
\label{charadjoint}
Let $X=(\mathbb{C}^n, \| \cdot \|)$ be an adjoint complex polytope space with the norm given for $x \in \mathbb{C}^n$ as
$$\|x\| = \max_{1 \leq j \leq N} |f_j(x)|,$$
where $f_1, \ldots, f_N: \mathbb{C}^n \to \mathbb{C}$ are non-zero functionals. Let $Y \subseteq X$ be a subspace and let $x \in X \setminus Y$ be a vector. Then $0 \in Y$ is a best approximation for $x$ in $Y$ if and only if, for every $y \in Y$ there exists $1 \leq j \leq N$ such that $|f_j(x)|=\|x\|$ and $\re \left (\overline{f_j(x)}f_j(y) \right ) \leq 0$. Moreover, $0$ is a unique best approximation for $x$ in $Y$ if and only if, for every non-zero $y \in Y$ there exists $1 \leq j \leq N$ such that $|f_j(x)|=\|x\|$, $\re \left ( \overline{f_j(x)}f_j(y) \right ) \leq 0$ and $f_j(y) \neq 0$.
\end{lem}

\begin{proof}
If for a given $y \in Y$ there exists $1 \leq j \leq N$ such that $|f_j(x)|=\|x\|$ and $\re \left ( \overline{f_j(x)}f_j(y) \right ) \leq 0$, then
$$\|x-y\|^2 \geq |f_j(x-y)|^2 = |f_j(x)|^2 - 2\re\left ( \overline{f_j(x)}f_j(y) \right ) + |f_j(y)|^2 \geq |f_j(x)|^2= \|x\|^2,$$
which proves that $\|x-y\| \geq \|x\|$ and hence $0$ is a best approximation for $x$ in $Y$. Moreover, the inequality is strict if $f_j(y) \neq 0$. Thus, if this holds for every non-zero $y \in Y$, then $0$ is a unique best approximation for $x$ in $Y$.

On the other hand, if $0$ is a best approximation for $x$, then by the previous lemma, there exists a functional $f \in B_{X^*}$ such that $f(x)=\|x\|$ and $f|_{Y} \equiv 0$. As $B_{X^*} = \absconv \{f_1, \ldots, f_N\}$ we can write $f$ as a convex combination of $f_1, \ldots, f_N$ multiplied with some complex scalars of modulus $1$. Without loss of generality let us assume that $f = a_1f_1 + \ldots + a_mf_m$, where $a_i \in \mathbb{C}$ are non-zero and $|a_1| + \ldots + |a_m|=1$. Then from the condition $f(x)=\|x\|$ and the triangle inequality it easily follows that $f_j(x)=\overline{a_j}\|x\|$ for every $1 \leq j \leq m$. Moreover, the condition $f|_Y \equiv 0$ implies that for every $y \in Y$ we have
$$0  = \re(f(y)) = \re(a_1f_1(y) + \ldots + a_mf_m(y))$$
$$=\|x\|\left ( \re\left ( \overline{f_1(x)}f_1(y) \right ) + \ldots + \re\left ( \overline{f_m(x)}f_m(y) \right )\right ),$$
and hence for at least one index $1 \leq j \leq m$ the inequality $\re(\overline{f_j(x)}f_j(y)) \leq 0$ must be true. 

To finish the proof, let us now assume additionally that $0$ is a unique best approximation for $x$ in $Y$. For the sake of a contradiction, we suppose that there exists a non-zero $y \in Y$ such that for every index $1 \leq j \leq N$ with $|f_j(x)|=\|x\|$ we have either  $\re(\overline{f_j(x)}f_j(y)) > 0$ or $f_j(y)=0$. Let 
$$M = \max \{|f_j(x)|: \ 1 \leq j \leq N, \ |f_j(x)|<\|x\| \}$$
(we take $M=0$ if $|f_j(x)|=\|x\|$ for every $1 \leq j \leq N)$. Then clearly $M < \|x\|$. Let $t$ be a positive real number which is not greater than $\frac{\|x\|-M}{\|y\|}>0$ and also not greater than every element of a set 
$$\left \{ \frac{2\re \left ( \overline{f_j(x)}f_j(y) \right )}{|f_j(y)|^2}: \ 1 \leq j \leq N, \ |f_j(x)|=\|x\|, \ \re\left ( \overline{f_j(x)}f_j(y) \right )>0 \right \}.$$ 
If now $1 \leq j \leq N$ is index such that $|f_j(x)|=\|x\|$ and $f_j(y)=0$ then obviously $|f_j(x-ty)|=\|x\|$. If $|f_j(x)|=\|x\|$ and $\re \left (\overline{f_j(x)}f_j(y) \right )>0$, then by the choice of $t$ we have
$$\|f_j(x-ty)\|^2 = |f_j(x)|^2 - 2t\re \left ( \overline{f_j(x)}f_j(y) \right ) + t^2|f_j(y)|^2 \leq |f_j(x)|^2=\|x\|^2.$$
Finally, if $|f_j(x)| < \|x\|$ then again by the choice of $t$
$$|f_j(x-ty)| \leq M + t\|y\| \leq  M + \|x\|-M=\|x\|.$$
It follows that $\|x-ty\|=\max_{1 \leq j \leq N} |f_j(x-ty)| \leq \|x\|$, which contradicts the assumption that $0$ is a unique best approximation for $x$ and the conclusion follows.
\end{proof}

Using the characterization above it is quite simple to establish that in an adjoint complex polytope norm, the uniqueness of a best approximation in $1$-dimensional subspaces automatically implies a $2$-strong uniqueness, similarly like in Theorem \ref{twrsudolskiwojcik}.

\begin{obs}
\label{obs1dim}
Let $X=(\mathbb{C}^n, \| \cdot \|)$ be an adjoint complex polytope space. Let $Y \subseteq X$ be a $1$-dimensional subspace and let $x \in X \setminus Y$ be a vector, for which $y_0 \in Y$ is a unique best approximation for $x$ in $Y$. Then $y_0$ is a $2$-strongly unique best approximation for $x$ in $Y$.
\end{obs}

\begin{proof}

By considering $x-y_0$ instead of $x$ we can assume that $y_0=0$. Let $f_1, \ldots, f_N \in B_{X^*}$ be functionals such that $B_{X^*} = \absconv \{f_1, \ldots, f_N\}$ and suppose that $Y = \lin \{y\}$ for some vector $y \in \mathbb{C}^n$ with $\|y\|=1$. Without loss of generality we can assume that $|f_j(x)|=\|x\|$ for $1 \leq j \leq m$, while $|f_j(x)|<\|x\|$ for $m+1 \leq j \leq N$ (possibly $m=N$). We want to prove that there exists a constant $r>0$ such that
$$\|x-\alpha t y\|^2 \geq \|x\|^2 + rt^2,$$
for every scalar $\alpha \in \mathbb{C}$ with $|\alpha|=1$ and $t \geq 0$. By Lemma \ref{charadjoint}, for a fixed $\alpha \in \mathbb{C}$ of modulus $1$ there exists $1 \leq j \leq m$ such that $\re\left ( \overline{f_j(x)}f_j(\alpha y) \right ) \leq 0$ and $f_j(y) \neq 0$. Then
$$\|x - \alpha t y\|^2 \geq |f_j(x - \alpha t y)|^2 = \|x\|^2 - 2t \re\left ( \overline{f_j}(x)f_j(\alpha y) \right ) + t^2 |f_j(y)|^2  \geq x^2 + t^2|f_j(y)|^2.$$
This shows that the desired inequality is true with the constant $r>0$ defined as
$$r= \min\{|f_j(y)|: \ 1 \leq j \leq m \text { and } f_j(y) \neq 0\}.$$
\end{proof}

It turns out, that a property like in the observation above, can fail already for two-dimensional subspaces. A counterexample can be found in the most basic example of an adjoint complex polytope norm, that is $\ell_{\infty}$ norm.

\begin{example}
\label{counterexample}
Let $X = (\mathbb{C}^4, \| \cdot \|_{\infty})$ and let $Y \subseteq X$ be a two-dimensional subspace defined as the span of $u, v \in \mathbb{C}^4$, where
$$u = (0, 0, -i, i) \quad \text { and } \quad  v = (i, -i, 1-i, 1+i).$$
We shall verify that $0$ is a unique best approximation in $Y$ for a unit vector $x=(1, 1, 1, 1)$. Indeed, let us take $y = (y_1, y_2, y_3, y_4) \in Y$, $y \neq 0$ of the form $y=au+bv$, where $a, b \in \mathbb{C}$. If $b=0$, then $a \neq 0$ and hence
$$\|x-y\|_{\infty} \geq \max \{|1-y_3|, |1-y_4| \} =  \max\{|1-ai|, |1+ai|\}>1.$$
On the other hand, for $b \neq 0$ we can estimate
$$\|x-y\|_{\infty} \geq \max \{|1-y_1|, |1-y_2|\} = \max \{|1-bi|, |1+bi| \}>1.$$
Thus $\|x-y\|_{\infty}>1$ for $y \in Y$, $y \neq 0$.

Now we will prove that $0$ is not a $2$-strongly unique best approximation for $x$ in $Y$. For an integer $n \geq 1$ let us define a vector $y^n \in Y$ as 
$$y^n=  \left (\frac{1}{n} - \frac{1}{n^2} \right )u + \frac{1}{n^2} v = \left ( \frac{i}{n^2}, \frac{-i}{n^2}, \frac{1}{n^2} - \frac{i}{n} , \frac{1}{n^2} + \frac{i}{n}  \right ) .$$ 
Clearly $\|y^n\|_{\infty}= \sqrt{\frac{1}{n^2} + \frac{1}{n^4}}$ and
$$x-y^n = \left ( 1 - \frac{i}{n^2}, 1 + \frac{i}{n^2}, 1 - \frac{1}{n^2} + \frac{i}{n}, 1 - \frac{1}{n^2} - \frac{i}{n} \right ).$$
Since 
$$\left | 1 - \frac{1}{n^2} \pm \frac{i}{n} \right | = \sqrt{1 - \frac{2}{n^2} + \frac{1}{n^4} + \frac{1}{n^2}} = \sqrt{1 - \frac{1}{n^2} + \frac{1}{n^4}} < 1$$
it follows that
$$\|x-y^n\|_{\infty} = \left  |1 \pm \frac{i}{n^2} \right | =  \sqrt{1 + \frac{1}{n^4}}.$$
Hence
$$\frac{\|x-y^n\|_{\infty}^2 - \|x\|_{\infty}^2}{\|y^n\|^2_{\infty}} = \frac{1 + \frac{1}{n^4} - 1}{\frac{1}{n^2} + \frac{1}{n^4}} = \frac{\frac{1}{n^4}}{\frac{1}{n^2} + \frac{1}{n^4}} = \frac{1}{n^2+1} \to 0.$$
We conclude that there is no constant $r>0$ such that $\|x-y^n\|_{\infty}^2 \geq \|x\|_{\infty}^2 + r \|y^n\|_{\infty}^2$ for every $n$ and the proof is finished.

\end{example}

We remark that by adding zeroes, the previous example can be obviously carried over to every $\ell_{\infty}^n$ space for $n \geq 4$. Moreover, it is easy to check that for the constructed sequence $(y^n)$, the smallest $\alpha$ for which there is no contradiction with the assumption of $\alpha$-strong uniqueness is $\alpha=4$. We do not know if there is anything special about $\alpha=4$ in general. It is a natural question if for every $n \geq 1$ there exists $\alpha(n) \geq 1$ such that for every subspace $Y \subseteq \ell_{\infty}^n$ and every vector $x \in X \setminus Y$ with a unique best approximation in $Y$, this best approximation is also $\alpha(n)$-strongly unique.

We turn our attention to the case of complex polytope spaces, which seems to be more difficult to handle, as in this case there is no convenient formula for the norm. However, also in this case we provide a characterization of an element of a best approximation and of its uniqueness. Before that we need a simple auxiliary result.

\begin{lem}
\label{ineqcompl}
For any two complex numbers $x, y \in \mathbb{C}$ such that $x \neq 0$ we have
$$|x-y| \geq |x| - \frac{\re(\overline{x}y)}{|x|}.$$
Moreover, if the equality holds then $\im \left ( \frac{\overline{x}}{|x|}y \right ) = 0.$
\end{lem}
\begin{proof}
If the right-hand side is negative, then there is nothing to prove as the left-hand side is non-negative. Otherwise, after squaring we get an equivalent inequality
$$|x-y|^2 = |x|^2 - 2\re(\overline{x}y) + |y|^2 \geq |x|^2 - 2\re(\overline{x}y) + \frac{\re(\overline{x}y)^2}{|x|^2},$$
which rewrites simply as
$$|y| \geq \left | \re\left ( \frac{\overline{x}}{|x|}y \right ) \right |.$$
This inequality is obviously true as $\frac{\overline{x}}{|x|}$ is a complex number of modulus $1$. Moreover, if the equality holds then $\frac{\overline{x}}{|x|}y$ is a real number and thus $\im \left ( \frac{\overline{x}}{|x|}y  \right )=0$.
\end{proof}

Below we provide a characterization similar to Lemma \ref{charadjoint}, but for the complex polytope norms. To our best knowledge such characterization was known only in the case of the most obvious example of a complex polytope norm, that is $\ell_1$ norm. For a general space $L_1(\mu)$, a characterization of an element of best approximation is known for many years (see for example Section $1.12$ in \cite{singer} or Section $5$ in \cite{survey}).

\begin{lem}
\label{lemchar}
Let $X = (\mathbb{C}^n, \| \cdot \|)$ be a complex polytope space with an essential system of vertices $\{ u_1, \ldots, u_N \} \subseteq \mathbb{C}^n$. Suppose that $Y \subseteq X$ is a subspace and $x \in X \setminus Y$ is a vector with a representation $x = a_1u_1 + \ldots a_Nu_N$ satisfying $\|x\|=|a_1| + \ldots + |a_N|.$ Let us suppose additionally that $1 \leq m \leq N$ is such that $a_j \neq 0$ for $1 \leq j \leq m$, while $a_j=0$ for $m+1 \leq j \leq N$ (possibly $m=N$ when all of the coordinates are non-zero). Then $0$ is a best approximation for $x$ in $Y$ if and only if, for every $y \in Y$ and every representation $y=b_1u_1 + \ldots + b_Nu_N$ we have
\begin{equation}
\label{ineqcond}
\left | \sum_{j=1}^{m} \frac{\overline{a_j}}{|a_j|} b_j \right | \leq \sum_{j=m+1}^{N} |b_j|.
\end{equation}
Moreover, $0$ is a unique best approximation for $x$ in $Y$ if and only if, the additional condition holds: if $y=b_1u_1 + \ldots + b_Nu_N \in Y$, $y \neq 0$ is such that
$$ \sum_{j=1}^{m} \frac{\overline{a_j}}{|a_j|} b_j = \sum_{j=m+1}^{N} |b_j|,$$
then there exists $1 \leq j \leq m$ such that $\im(\overline{a_j}b_j) \neq 0$.

\end{lem}

We note that in the case of $m=N$ (all coordinates of $x$ are non-zero) the inequality (\ref{ineqcond}) reads just as
$$\sum_{j=1}^{N} \frac{\overline{a_j}}{|a_j|} b_j = 0 .$$

\begin{proof}
Let us first assume that $0$ is a best approximation for $x$ in $Y$ and let us fix a non-zero $y=b_1u_1 + \ldots + b_Nu_N \in Y$. For every $t>0$ we have
$$\frac{\|x-ty\| - \|x\|}{t} \geq 0.$$
Since $x-ty = \sum_{j=1}^{N} (a_j - tb_j)u_j$ we have $\|x-ty\| \leq \sum_{j=1}^{N}|a_j - tb_j|$. Therefore
$$\sum_{j=1}^{N} \frac{ |a_j-tb_j| - |a_j|}{t} \geq 0.$$
We note that
$$ \sum_{j=1}^{N} \frac{ |a_j-tb_j| - |a_j|}{t} = \sum_{j=1}^{m} \frac{ |a_j-tb_j| - |a_j|}{t} + \sum_{j=m+1}^{N} |b_j|$$
$$=\sum_{j=1}^{m} \frac{ |a_j-tb_j|^2 - |a_j|^2}{t(|a_j-tb_j| + |a_j|)} + \sum_{j=m+1}^{N} |b_j| = \sum_{j=1}^{m} \frac{ -2\re(\overline{a_j}b_j) + t|b_j|^2}{|a_j-tb_j| + |a_j|} + \sum_{j=m+1}^{N} |b_j| \geq 0.$$
Since $|a_j| \neq 0$ for $1 \leq j \leq m$ we can take the limit with $t \to 0^+$ to get the inequality
$$\re \left ( \sum_{j=1}^{m} \frac{\overline{a_j}}{|a_j|} b_j  \right ) \leq \sum_{j=m+1}^{N} |b_j|.$$
The inequality \eqref{ineqcond} follows now by multiplying $y$ by a scalar $\alpha \in \mathbb{C}$ of modulus $1$ such that the number on the left hand side is a non-negative real number.

Now assuming that the inequality \eqref{ineqcond} holds, we shall establish that $0$ is a best approximation for $x$ in $Y$. Thus, we need to prove that $\|x-y\| \geq \|x\|$ for any $y \in Y$. Let us take a representation
$$x-y = c_1u_1 + \ldots + c_Nu_N,$$
such that $\|x-y\| = |c_1| + \ldots + |c_N|$. Obviously we then have $y = b_1u_1 + \ldots + b_Nu_N$, where $b_j = a_j - c_j$ for $1 \leq j \leq N$ and the condition \eqref{ineqcond} holds for $b_1, \ldots, b_N$. Thus, by Lemma \ref{ineqcompl} we have
$$\|x-y\| = \sum_{j=1}^{N} |c_j| = \sum_{j=1}^{m} |a_j-b_j| + \sum_{j=m+1}^{N} |b_j| \geq \sum_{j=1}^{m} \left ( |a_j| - \frac{\re(\overline{a_j}b_j)} {|a_j|} \right ) + \sum_{j=m+1}^{N} |b_j|$$
$$= \sum_{j=1}^{m} |a_j| + \sum_{j=m+1}^{N} |b_j| - \re \left ( \sum_{j=1}^{m} \frac{\overline{a_j}}{|a_j|} b_j  \right ) \geq \sum_{j=1}^{m} |a_j| = \|x\|. $$
This proves that $0$ is a best approximation for $a$. Moreover, from the estimate above and the equality condition in Lemma \ref{ineqcompl} it follows immediately that if equality holds for some non-zero $y \in Y$, then we must have $\re \left ( \sum_{j=1}^{m} \frac{\overline{a_j}}{|a_j|} b_j  \right ) = \sum_{j=m+1}^{N} |b_j|$ and $\im(\overline{a_j}b_j)=0$ for every $1 \leq j \leq m$. Clearly the first condition can be written also as $ \sum_{j=1}^{m} \frac{\overline{a_j}}{|a_j|} b_j = \sum_{j=m+1}^{N} |b_j|$.

We are left with proving that if the condition \eqref{ineqcond} holds, but there exists non-zero $y =b_1u_1 + \ldots + b_Nu_N \in Y$ satisfying $\sum_{j=1}^{m} \frac{\overline{a_j}}{a_j} b_j = \sum_{j=m+1}^{N} |b_j|$ and $\im(\overline{a_j}b_j)=0$ for every $1 \leq j \leq m$, then $0$ is not a unique best approximation for $x$. We will prove that in this case $\|x-ty\|=\|x\|$ for sufficiently small and real $t>0$. Indeed, we have
$$\|x-ty\| \leq \sum_{j=1}^{N} |a_j - tb_j| = \sum_{j=1}^{m} |a_j - tb_j| + \sum_{j=m+1}^{N} t|b_j|.$$
We note that by the proof of Lemma \ref{ineqcompl} the inequality 
$$|a_j-tb_j| \geq |a_j| - t\frac{\re\left ( \overline{a_j}b_j \right )}{|a_j|}$$
will be an equality when $\im(\overline{a_j}b_j)=0$ and $t>0$ is such that $|a_j| \geq t|b_j|$. We note that such $t>0$ is always possible to choose for every $1 \leq j \leq m$, as in this range we have $|a_j| \neq 0$. Hence for such $t>0$ we have
$$\sum_{j=1}^{m} |a_j - tb_j| + \sum_{j=m+1}^{N} t|b_j| = \sum_{j=1}^{m} \left ( |a_j| - t\frac{\re\left ( \overline{a_j}b_j \right )}{|a_j|} \right ) + t \sum_{j=m+1}^{N} |b_j|$$
$$=\sum_{j=1}^{m} |a_j| - t \sum_{j=1}^{m} \frac{\re\left ( \overline{a_j}b_j \right )}{|a_j|} + t \sum_{j=m+1}^{N} |b_j|$$
$$=\sum_{j=1}^{m} |a_j| - t \sum_{j=1}^{m} \frac{ \overline{a_j}b_j}{|a_j|} + t \sum_{j=m+1}^{N} |b_j| = \sum_{j=1}^{m} |a_j| = \|x\|.$$
Therefore, we have $\|x-ty\| \leq \|x\|$. Because also the opposite inequality is true, it follows that $\|x-ty\|=\|x\|$ and therefore $0$ is not a unique best approximation for $x$. This concludes the proof.
\end{proof}

We do not know if in complex polytope spaces a uniqueness of best approximation automatically implies a $2$-strong uniqueness, even for $1$-dimensional subspaces. It seems quite likely that a counterexample like Example \ref{counterexample} exists also in some complex polytope space. We can prove however, that a counterexample can not be found in the space $\ell_1^n$. In other words, an analogue of Theorem \ref{twrsudolskiwojcik} is true in the complex $\ell_1^n$ space (with a $2$-strong uniqueness). To our best knowledge, such a result was not known previously. This may be quite surprising considering a large volume of research devoted to studying approximation in $L_1$ spaces (see for instance \cite{l1book} for a whole book covering this topic) and the fact that a characterization of an element of best approximation in $\ell_1^n$ is known for a long time. In the case of the $\ell_1^n$ space, the characterization given in Lemma \ref{lemchar} is much more convenient, as every vector has a unique representation in the canonical unit basis. We need first the following technical variant of Lemma \ref{ineqcompl}.

\begin{lem}
\label{ineqcompl2}
Let $x, y \in \mathbb{C}$ be complex numbers such that $x \neq 0$. Then for every real number $t$ satisfying $|ty| \leq \frac{|x|}{2}$  we have
$$|x-ty| \geq |x| - t \frac{\re(\overline{x}y)}{|x|} + t^2 \frac{\im(\overline{x}y)^2}{4|x|^3}.$$
\end{lem}

\begin{proof}
Since the left-hand side is non-negative, there is nothing to prove if the right-hand side is negative. We can therefore square the desired inequality to get
$$|x|^2 - 2t \re(\overline{x}y) + t^2|y|^2 \geq |x|^2 - 2t \re(\overline{x}y) + \left ( \frac{\re(\overline{x}y)^2}{|x|^2} + \frac{\im(\overline{x}y)^2}{2|x|^2} \right ) t^2 + At^3 + Bt^4,$$
where
$$A= - \frac{\re(\overline{x}y)\im(\overline{x}y)^2}{2|x|^4} \quad \text{ and } \quad B=\frac{\im(\overline{x}y)^4}{16|x|^6}.$$
Moreover, we note that 
$$\frac{\re(\overline{x}y)^2}{|x|^2} + \frac{\im(\overline{x}y)^2}{2|x|^2} = \frac{2\re(\overline{x}y)^2 + \im(\overline{x}y)^2}{2|x|^2} = \frac{ 2|xy|^2 - \im(\overline{x}y)^2}{2|x|^2} = |y|^2 - \frac{\im(\overline{x}y)^2}{2|x|^2}.$$
Therefore, the desired inequality is equivalent to
$$\frac{\im(\overline{x}y)^2}{2|x|^2}t^2 \geq At^3 + Bt^4.$$
Since the inequality is obviously true for $t=0$ or $\im(\overline{x}y)=0$ we can simplify it to
$$8|x|^4 \geq -8t \re(\overline{x}y)|x|^2 + t^2 \im(\overline{x}y)^2.$$
Clearly
$$\left | -8t \re(\overline{x}y)|x|^2 + t^2 \im(\overline{x}y)^2 \right | \leq 8|t||x|^3|y| + |t|^2|x|^2|y|^2.$$
Therefore, the desired inequality is true if both inequalities
$$8|t||x|^3|y| \leq 4|x|^4 \quad \text{ and } \quad |t|^2|x|^2|y|^2 \leq 4|x|^4$$
are satisfied and they indeed are when $|ty| \leq \frac{|x|}{2}$. This concludes the proof.
\end{proof}

We shall also need the following lemma, in which we observe that in the finite-dimensional setting, it is enough to establish the property of strong uniqueness locally around $0$. We recall here that $1$-strong uniqueness implies $\alpha$-strong uniqueness for any $\alpha \geq 1$.

\begin{lem}
\label{lemcompact}
Let $X$ be a normed space over $\mathbb{K}$ and let $Y \subseteq X$ be a finite-dimensional subspace. If $x \in X \setminus Y$ is a vector such that $0$ is a unique best approximation for $x$, then for every $d>0$ there exists a constant $r_d>0$ such that
$$\|x-y\| \geq \|x\| + r_d\|y\|$$
for every $y \in Y$ with $\|y\| \geq d$.
\end{lem}

\begin{proof}
We note that if $\| y\| \geq 4\|x\|$, then
$$\|x-y\| \geq \|y\| - \|x\| \geq \|x\| + \frac{\|y\|}{2},$$
so in this case the desired inequality holds with $r_d = \frac{1}{2}$. In particular, we are done if $d \geq 4 \|x\|$. Otherwise, let $B \subseteq Y$ be defined as $B = \{ y \in Y: d \leq \|y\| \leq 4\|x\| \}$. Since $Y$ is finite-dimensional, the set $B$ is compact and hence a strictly positive function
$$\frac{\|x-y\|-\|x\|}{\|y\|}$$
attains its minimum $c_d>0$ on $B$. Taking $r_d$ as $r_d=\min \left \{ \frac{1}{2}, c_d \right \}$ it follows that the desired inequality is satisfied for all $\|y\| \geq d$.
\end{proof}

\begin{twr}
\label{twrl1}
Let $Y \subseteq \ell_{1}^n$ be a subspace of the complex $\ell_1^n$ space. Suppose that $y_0 \in Y$ is a unique best approximation in $Y$ for a vector $x \in \ell_{1}^n$. Then $y_0$ is a $2$-strongly unique best approximation for $x$ in $Y$.
\end{twr}

\begin{proof}

By considering a vector $x - y_0$ instead of $x_0$ we can assume that $y_0=0$. Without loss of generality we may further assume that $x_j \neq 0$ for $1 \leq j \leq m$ (where $m \leq n$) and $x_j=0$ for $m+1 \leq j \leq n$ (possibly $m=n$). We shall establish the inequality
\begin{equation}
\label{l1teza}
\|x-y\|_1 \geq \|x\|_1 + r\|y\|^2_1,
\end{equation}
for every $y \in Y$ and some constant $r>0$. This will be clearly enough, as after taking the square of both sides we shall get
$$\|x-y\|_1^2 \geq \|x\|_1^2 + (2\|x\|_1r)\|y\|^2_1.$$
In order to establish \eqref{l1teza} we shall use Lemma \ref{lemchar}. In the case of the space $\ell_1^n$, a representation of every vector is unique (as the essential system of vertices is just the canonical unit basis) and hence, from the fact that $0$ is a unique best approximation for $x$ in $Y$, it follows that
$$\left | \sum_{j=1}^{m} \frac{\overline{x_j}}{|x_j|} y_j \right | \leq \sum_{j=m+1}^{n} |y_j| \quad \text{ for every } y \in Y,$$
and moreover, if the equality holds in the inequality above for some $y \neq 0$ then
$$\sum_{j=1}^{m} \im(\overline{x_j}y_j)^2 > 0.$$
Let $\mathcal{S}$ be  subset of the unit sphere $S_Y$ of $Y$ defined as
$$\mathcal{S} = \left \{ y \in S_Y: \  \left | \sum_{j=1}^{m} \frac{\overline{x_j}}{|x_j|} y_j \right | = \sum_{j=m+1}^{n} |y_j| \right \}.$$
We shall consider two cases: if the set $\mathcal{S}$ is empty or not.

Let us first assume that $\mathcal{S}$ is empty. In this case we shall prove that $0$ is actually a $1$-strongly unique best approximation for $x$ in $Y$. From the compactness of the unit sphere there exists a $\varepsilon>0$ such that for every $y \in S_Y$ we have
$$\sum_{j=m+1}^{n} |y_j| \geq \left | \sum_{j=1}^{m} \frac{\overline{x_j}}{|x_j|} y_j \right | + \varepsilon.$$
In that case, from Lemma \eqref{ineqcompl} it follows that for every $y \in Y$ we have
$$\|x-y\|_1 = \sum_{j=1}^{m} |x_j-y_j| + \sum_{j=m+1}^{n} |y_j| \geq \sum_{j=1}^{m} \left ( |x_j| - \frac{\re(\overline{x_j}y_j)}{|x_j|} \right ) + \sum_{j=m+1}^{n} |y_j| $$
$$= \sum_{j=1}^{m} |x_j| + \sum_{j=m+1}^{n} |y_j| - \re \left (\sum_{j=1}^{m} \frac{\overline{x_j}}{|x_j|}y_j \right )\geq \|x\|_1 + \varepsilon \|y\|_1.$$
This shows that $0$ is a $1$-strongly unique best approximation for $x$ in $Y$ with a constant $\varepsilon>0$.

Now let us assume that the set $\mathcal{S}$ is non-empty. Since it is closed and contained in the unit sphere of $Y$, it is a compact set. Therefore, there exists a constant $c>0$ such that for every $y \in \mathcal{S}$ we have
$$\sum_{j=1}^{m} \im(\overline{x_j}y_j)^2 \geq c.$$
Clearly, from a continuity there exists a $\varepsilon>0$ such that if $y \in \mathcal{S}_{\varepsilon}$, where
$$\mathcal{S}_{\varepsilon} = \left \{ y \in S_Y : \ \sum_{j=m+1}^{n} |y_j| \leq \left | \sum_{j=1}^{m} \frac{\overline{x_j}}{|x_j|} y_j \right | + \varepsilon \right \},$$
then
$$\sum_{j=1}^{m} \im(\overline{x_j}y_j)^2 \geq \frac{c}{2}.$$
If $y \in Y$, $y \neq 0$ is such that $\frac{y}{\|y\|_1} \not \in \mathcal{S}_{\varepsilon}$, then reasoning exactly in the same way as in the previous step we get that
$$\|x-y\|_1 \geq \|x\|_1 + \varepsilon\|y\|_1.$$

Let us now assume that $y  \in \mathcal{S}_{\varepsilon}$. Since $\|y\|_1 = 1$ we obviously have $|y_j| \leq 1$ for every $1 \leq j \leq n$. Hence, for a real $t$ such that $0 < t \leq \frac{1}{2} \min \{|x_1|, \ldots, |x_m|\}$ by Lemma \ref{ineqcompl2} we can now estimate
$$\|x-ty\|_1 = \sum_{j=1}^{m} |x_j - ty_j| + \sum_{j=m+1}^{n} |ty_j|$$
$$\geq \sum_{j=1}^{m} \left ( |x_j| -  t \frac{\re(\overline{x_j}y_j)}{|x_j|} + t^2 \frac{\im(\overline{x_j}y_j)^2}{4|x_j|^3} \right ) + |t|\sum_{j=m+1}^{N} |y_j|$$
$$=\|x\|_1 + t \left (\sum_{j=m+1}^{n} |y_j| - \sum_{j=1}^{m} \frac{\re(\overline{x_j}y_j)}{|x_j|} \right ) + t^2\sum_{j=1}^{m}\frac{\im(\overline{x_j}y_j)^2}{4|x_j|^3} \geq \|x\|_1 + \frac{c}{8\|x\|_{\infty}^3}t^2.$$
Altogether, we have obtained $0$ is a $2$-strongly unique best approximation for $x$ in $Y$, when $\|y\|_{1} \leq \frac{1}{2} \min \{|x_1|, \ldots, |x_m|\}$. From Lemma \ref{lemcompact} it follows now that $0$ is a $2$-strongly unique best approximation for $x$ in $Y$ and the conclusion follows.
\end{proof}

\begin{remark}
The proof above implies an equivalent condition for $1$-strong uniqueness of a best approximation in the space $\ell_1^n$, namely if the set $\mathcal{S}$ defined in the proof is empty or not. We note that such condition is for a $1$-strong uniqueness is nothing new as it is known for a long time -- see Theorems $5.2$ and $5.3$ in \cite{survey}. However, we are not aware of any previous results concerned with a $2$-strong uniqueness in the complex $\ell_1^n$ space.
\end{remark}

Below we provide an example showing that a $2$-strong uniqueness can not be improved to an $\alpha$-strong uniqueness for any $\alpha<2$ for some subspaces of the complex $\ell_1^n$ space. This is a rather usual situation in the complex setting. However, later we shall see that for a broad class of subspaces of $\ell_1^n$ the order can actually be improved to a $1$-strong uniqueness (see Theorem \ref{twrl12}).

\begin{example}
Let $n \geq 4$ and let $Y = \lin\{y\} \subseteq \ell_1^n$ be a $1$-dimensional subspace, where $y = (1, -1, i, -i, 0 \ldots, 0)$. Moreover, let us take $x = (1, 1, 1, 1, 0 \ldots, 0)$. Since $y_1 + y_2 + y_3 + y_4 =0$, it follows Lemma \ref{lemchar} that $0$ is a best approximation for $x$ in $Y$. Moreover, since $ty$ is not a real vector for any $t \neq 0$, the condition for uniqueness from that lemma gives us that $0$ is a unique best approximation for $x$ in $Y$. Therefore, it follows from Theorem \ref{twrl1}, that $0$ is a $2$-strongly unique best approximation for $x$. We shall prove that it is not an $\alpha$-strongly unique best approximation for any $\alpha<2$. Indeed, if that would be the case, then for any real $t>0$ we would have
$$\|x-ty\|^{\alpha}_1 \geq \|x\|_1^{\alpha} + rt^{\alpha}$$
for some constant $r>0$. This rewrites as
$$\left ( |1-t| + |1+t| + |1-ti| + |1+ti| \right )^{\alpha} \geq 4^{\alpha} + rt^{\alpha}$$
or assuming that $ t \in (0, 1)$
$$\frac{\left (2 + 2\sqrt{1+t^2} \right )^{\alpha} - 4^{\alpha}}{t^{\alpha}} \geq r,$$
which is equivalent to
$$\frac{\left (1 + \sqrt{1+t^2} \right )^{\alpha} - 2^{\alpha}}{t^{\alpha}} \geq \frac{r}{2^{\alpha}}.$$
To prove that this inequality is impossible for small $t>0$, we shall calculate the limit of the left-hand side with $t \to 0^{+}$ using the L'Hospital's rule. Indeed, since $\alpha<2$ we have
$$\lim_{t \to 0^{+}} \frac{\left (1 + \sqrt{1+t^2} \right )^{\alpha} - 2^{\alpha}}{t^{\alpha}} = \lim_{t \to 0^{+}} \frac{t}{\sqrt{t^2+1}} \cdot \frac{\left (1 + \sqrt{1+t^2} \right )^{\alpha-1}}{t^{\alpha-1}}$$
$$=\lim_{t \to 0^{+}} \frac{t^{2-\alpha}}{\sqrt{t^2+1}} \cdot \left (1 + \sqrt{1+t^2} \right )^{\alpha-1} = 0.$$
This gives us a desired contradiction and in conclusion $0$ is not an $\alpha$-strongly unique best approximation for any $\alpha<2$.
\end{example}

\section{$2$-strong uniqueness of a best approximation in real subspaces of complex polytope spaces and their duals}
\label{sectionreal}

In the previous section we saw that a property like in Theorem \ref{twrsudolskiwojcik} for real polytope spaces does not transfer to adjoint complex polytope spaces (Example \ref{counterexample}). While the situation is less clear for complex polytope spaces, it is quite possible that a similar counterexample could be found. It is therefore natural to ask about some additional conditions, which would guarantee that a unique best approximation is automatically $2$-strongly unique for complex polytope spaces or for their duals. In this section we provide such a condition, that holds for a broad class of complex polytope or adjoint complex polytope spaces and their subspaces. We shall call a subspace $V \subseteq \mathbb{C}^n$ a \emph{real} subspace if it has a basis consisting of real vectors. Our goal is to prove that if a complex polytope norm has an essential system of vertices consisting of real vectors or an adjoint complex polytope norm has an essential system of facets consisting of real functionals, then a property like in Theorem \ref{twrsudolskiwojcik} holds true for all real subspaces. Therefore, what we shall consider here is, in a certain sense, an intermediate case between the real and complex settings.

By a real functional we shall always mean a functional, which transforms real vectors to real scalars, or equivalently, that it has real coordinates when treated as a vector. We start with some basic algebraic properties of the real subspaces. For a vector $v=(v_1, \ldots, v_n) \in \mathbb{C}^n$ by $\re(v)$ we shall always understand the vector $(\re(v_1), \ldots, \re(v_n)) \in \mathbb{R}^n$ and similarly $\im(v) = (\im(v_1), \ldots, \im(v_n)) \in \mathbb{R}^n$.

\begin{lem}
\label{realcond}
Let $V \subseteq \mathbb{C}^n$ be a subspace of a dimension $1 \leq d \leq n$. Then the following conditions are equivalent:
\begin{enumerate}
\item $V$ is a real subspace,
\item There exist vectors $f_1, \ldots, f_{n-d} \in \mathbb{R}^n$ such that $V = \{ v \in \mathbb{C}^{n}: \ \langle v, f_j \rangle = 0 \text { for every } 1 \leq j \leq n-d\}$.
\item For every $v \in V$ we have $\re(v) \in V$.
\end{enumerate}

\end{lem}

\vskip 0.2in

\begin{proof}
We first note that the first condition implies the second one. Indeed, if $v_1, \ldots, v_d \in V$ is a real basis of $V$, then we can consider a real subspace spanned by vectors $v_j$ in $\mathbb{R}^n$ and find a basis $f_1, \ldots, f_{n-d} \in \mathbb{R}^n$ of an orthogonal complement. Then $\langle v_k, f_j \rangle = 0$ for all $j, k$. Now if $v \in V$ is in the form $v = a_1v_1 + \ldots + a_dv_d$ for some scalars $a_j \in \mathbb{C}$ then by the linearity on the first coordinate of an inner product we immediately get that $\langle v, f_j \rangle = 0$ for any $1 \leq j \leq n-d$ .

Now we shall prove that the second condition implies the third. Indeed if $v \in V$, then by the assumption $\langle v, f_j \rangle = 0 $ for any $1 \leq j \leq n-d$. However, by comparing the real parts of this equality we get $\langle \re(v), f_j \rangle = 0$, so that also $\re(v) \in V$.

We are left with with proving that the last condition implies the first one. Let $v_1, \ldots, v_d \in \mathbb{C}^n$ be any basis of $V$. Then by assumption we have that $u_j=\re(v_j) \in V$ and $w_j=\re(-i v_j) = \im(v_j) \in V$ for any $1 \leq j \leq d$. However, the span over $\mathbb{C}$ of $u_1, \ldots, u_d, w_1, \ldots, w_d$ clearly contains $V$ as $v_j = u_j + i w_j$ for any $1 \leq j \leq d$. Therefore, from these $2d$ real vectors we can choose a basis of $V$ and the conclusion follows.

\end{proof}

In the lemma below we establish a property like in Theorem \ref{twrsudolskiwojcik}, but in a quite particular setting. We shall soon see in Theorem \ref{twrpolytope} that, in contrast with the real case (where by Theorem \ref{twrsudolskiwojcik} we have a $1$-strong uniqueness), a $2$-strong uniqueness cannot be improved to an $\alpha$-strong uniqueness for any $\alpha < 2$.

\begin{lem}
\label{lem1}
Let $Y \subseteq \ell_{\infty}^n$ be a real subspace of a complex $\ell_{\infty}^n$ space. Suppose that $0$ is a unique best approximation in $Y$ for a vector $x=(1, \ldots, 1) \in \ell_{\infty}^n$. Then $0$ is a $2$-strongly unique best approximation for $x$ in $Y$.
\end{lem}

\begin{proof}
We need to show that there exists a constant $r>0$ such that for any $y \in Y$ we have
\begin{equation}
\label{teza2strong}
\|x-y\|^2_{\infty} \geq 1 + r \|y\|_{\infty}^2.
\end{equation}

Since $Y$ is a real subspace, we have $\re(y)$ for $y \in Y$. In particular, since $0$ is also a unique best approximation for $x$ in a subspace of real parts of vectors from $Y$ (considered then as a subspace of the real space $\ell_{\infty}^n$), the zero vector is $1$-strongly unique for $x$ in this case by Theorem \ref{twrsudolskiwojcik}. Hence, there exists a constant $r_1 \in \left (0, 1 \right )$ such that for any $y \in Y$ we have
$$\|x-y\|_{\infty} \geq \|x-\re(y)\|_{\infty} \geq 1+ r_1\|\re(y)\|_{\infty},$$
so that
\begin{equation}
\label{ineq2strong}
\|x-y\|_{\infty}^2 \geq 1 + 2r_1\|\re(y)\|_{\infty}+ r_1^2\|\re(y)\|_{\infty}^2.
\end{equation}

Now, if $y \in Y$ is a vector such that
$$2 \|\re(y)\|_{\infty} + r_1 \|\re(y)\|_{\infty}^2 \geq \frac{r_1}{2} \|y\|_{\infty}^2,$$
then by \eqref{ineq2strong} the inequality \eqref{teza2strong} holds with $r = \frac{r_1^2}{2}$. Thus, let us suppose that $y \in Y$ does not satisfy the condition above. In particular we have
$$2 \|\re(y)\|_{\infty} \leq \frac{r_1}{2} \|y\|_{\infty}^2\quad \text{ and } \quad \|\re(y)\|_{\infty}^2 \leq \frac{1}{2} \|y\|_{\infty}^2.$$
Assume that $\|y\|_{\infty}=|y_j|=|a+bi|$ for some $1 \leq j \leq n$ and $a, b \in \mathbb{R}$. In particular
$$2|a| \leq 2 \|\re(y)\|_{\infty} \leq \frac{r_1}{2} \|y\|_{\infty}^2$$
and
$$a^2=|\re(y_j)|_{\infty}^2 \leq \|\re(y)\|_{\infty}^2 \leq \frac{1}{2} \|y\|_{\infty}^2 = \frac{a^2+b^2}{2},$$
which gives us $b^2 \geq a^2$ or equivalently $b^2 \geq \frac{1}{2} \|y\|_{\infty}^2$. Therefore
$$\|x-y\|_{\infty}^2 \geq |1-y_j|^2= |(1-a) - bi|^2 = (1-a)^2 + b^2 = 1 - 2a + a^2 + b^2$$
$$\geq 1 - 2|a| + b^2 \geq 1 - \frac{r_1}{2} \|y\|_{\infty}^2 + \frac{1}{2} \|y\|_{\infty}^2 = 1 + \frac{1-r_1}{2} \|y\|_{\infty}^2.$$
Altogether, we see that \eqref{teza2strong} holds for all $y \in Y$ with $r = \min \left ( \frac{r_1^2}{2}, \frac{1-r_1}{2} \right )$ and the proof is finished.
\end{proof}

In the following theorem we establish a $2$-strong uniqueness of a best approximation (like in Thereom \ref{twrsudolskiwojcik}) in a setting that is more general than just real subspaces of adjoint complex polytope spaces with a real essential system of facets, although this situation is an obvious example where all of the conditions are met.

\begin{twr}
\label{twr2strgeneral}
Let $X$ be a normed space over $\mathbb{C}$ and let $Y \subseteq X$ be a finite-dimensional subspace. Suppose that $x \in X \setminus Y$ is such a vector that $0$ is its unique best approximation in $Y$ and a set
$$
E(x) = \{ f \in \ext(S_{X^*}): f(x) = \|x\| \}
$$
is a finite set consisting of $N$ functionals $f_1, \ldots, f_N \in S_{X^*}$. Let $0 \leq M \leq \|x\|$ be defined as
$$
M= \sup\{|f(x)|: f \in \ext(S_{X^*} )\setminus |E|(x) \},
$$
where $ |E|(x) = \{ f \in \ext(S_{X^*}): |f(x)| = \|x\| \}$ (if $\ext(S_{X^*})\setminus |E|(x)$ is an empty set, then we take $M=0$). Let $Z \subseteq X$ be a subspace given as $Z=  \lin(Y \cup \{x\}).$ Assume that a subspace $F(Z) \subseteq \mathbb{C}^N$ defined as
$$ 
F(Z) = \{ (f_1(z), \ldots ,f_N(z)): \: z \in Z\}
$$
is a real subspace of $\mathbb{C}^N$. If the inequality $M < \|x\|$ holds, then $0$ is a $2$-strongly unique best approximation for $x$ in $Y$.
\end{twr}
\begin{proof}
We can assume that $\|x\|=1$, so that $M<1$. For $z \in Z$ let us define
$$\|z\|_0 = \max_{1 \leq j \leq N} |f_j(z)|.$$
We observe that $\|x\|_0=\|x\|=1$ and $\|z\| \geq \|z\|_0$ for any $z \in Z$. Moreover, we claim that $\| \cdot \|_0$ is in fact a norm on $Z$. Indeed, it is enough to check that $\|z\|_0 \neq 0$ for $z \neq 0$. Let us thus assume that there exists a non-zero vector $z \in Z$ such that for all $1 \leq j \leq N$ we have $f_j(z)=0$. We can write $z=ax+by$ for some $a, b \in \mathbb{C}$, which are not both zero and $y \in Y$. Clearly we must have $b \neq 0$. We note that since $M < 1$, if $t$ is a complex number with $|t|$ sufficiently small we have $|1-at|M + |b||t|\|y\| < 1$. In particular, for any $f \in \ext(S_{X^*})\setminus |E|(x)$ we have then
$$|f(x-tz)|=|f((1-at)x-by)| \leq |1-at|M + |b||t|\|y\| < 1.$$
Moreover, $|f_j(x-tz)|=|f_j(x)|=1=\|x\|$ for any $1 \leq j \leq N$. This means that $\|x-tz\| = \|(1-at)x - tby\|= 1$ for any $t \in \mathbb{C}$ with $|t|$ sufficiently small. However, it is clear that since $0$ is a unique best approximation for a vector $x$ in $Y$, then $0$ is also a unique best approximation for a vector $(1-at)x$ in $Y$ (which is of norm $|1-at|$). In particular, since $b \neq 0$ the element $tby$ is not a best approximation for $(1-at)x$ and it follows that $|1-at| < 1$. On the other hand, for any fixed complex number $a$, we can find a complex number $t$ with $|t|$ arbitrarily small and satisfying $|1-at|>1$. This gives us a contradiction, which proves that $\| \cdot \|_0$ is indeed a norm on $Z$.

By arguing in a similar way, we can prove that $0$ is still a unique best approximation for $x$ in $Y$, when the norm $\| \cdot \|_0$ is considered. Indeed, let us assume that for some non-zero $y \in Y$ we have $\|x-y\|_0 \leq \|x\|_0=\|x\|=1$. We note that if the inequality $\left | 1 - a \right | \leq 1$ holds for some $a \in \mathbb{C}$, then also $\left | 1 - a t \right | \leq 1$ for every real $t \in [0, 1]$. In particular, since for every $1 \leq j \leq N$ by assumption we have $|f_j(x-y)| \leq 1$, for any real $t \in [0, 1]$ we also have $|f_j(x-ty)| \leq 1$. If we now take $t>0$ such that $M+t\|y\|<1$ then again $|f(x-ty)| < 1$ for any $f \in \ext(S_{X^*})\setminus |E|(x)$. This means that $\|x-ty\|=1$ for $t>0$ sufficiently small and we obtain a contradiction with the fact that $0$ is a unique best approximation for $x$ in $Y$, when the norm $\| \cdot \|$ is considered.

Now let us consider a mapping $T: Z \to \ell_{\infty}^N$ defined as
$$T(z)=(f_1(z), \ldots f_N(z))$$
for $z \in Z$. It is clear that $T$ is a linear isometry between $(Z, \| \cdot \|_1)$ and a subspace of the complex $\ell_{\infty}^N$ space (considered with the standard norm), which by the assumption is a real subspace. Clearly $T(x)=(1, \ldots, 1)$, so by Lemma \ref{lem1} there exists a constant $r>0$ such that
$$\|x-y\|^2_0 = \|T(x)-T(y)\|^2_{\infty} \geq 1 + r\|T(y)\|_{\infty}^2 = 1 + r\|y\|^2_0.$$
It remains to observe that, since $Z$ is finite-dimensional, the norms $\| \cdot \|$ and $\| \cdot \|_0$ are equivalent, so that we have $\|z\|_0 \geq c \|z\|$ for some constant $c>0$ and all $z \in Z$. In particular
$$\|x-y\|^2 \geq \|x-y\|^2_0 \geq 1 + r\|y\|^2_0 \geq 1 + cr \|y\|^2$$
This shows that $0$ is $2$-strongly unique best approximation for $x$ in $Y$ with the constant $cr>0$ and the conclusion follows.
\end{proof} 

Below we provide a simple example illustrating that Theorem \ref{twr2strgeneral} can be applied to different situations, other than just the adjoint complex polytope spaces.

\begin{example}
Let $\mathbb{D} \subseteq \mathbb{C}$ be the unit disc and let $K=\mathbb{D} \cup \{1+i, 1-i, -1+i, -1-i\} \subseteq \mathbb{C}.$ In the space $C(K)$ of complex valued continuous functions equipped with the supremum norm, let $Y$ be the three-dimensional subspace spanned by the functions $1, \re(z), \im(z) \in C(K)$. Let $f \in C(K)$ be the real-valued function defined as $f(z) = \re(z) \cdot \im(z)$. In this case we have $\|f\|_K=1$ and $f$ attains its maximum exactly at the points of the form $\pm 1 \pm i$, while $|f(z)| \leq \frac{1}{2}$ for $z \in \mathbb{D}$ (so that $M=\frac{1}{2}$ using the notation of Theorem \ref{twr2strgeneral}). It is easy to verify that $0 \in Y$ is a unique best approximation for $f$ in $Y$, as $\max\{|f(z)-g(z)|: \ z = \pm 1 \pm i \}>1$ for any for non-zero $g \in Y$. Hence, by Theorem \ref{twr2strgeneral}, the zero function is a $2$-strongly unique best approximation for $f$.

\end{example}

The following lemma gives a simple but important property of complex polytope spaces with a real essential system of vertices and of adjoint complex polytope spaces with a real essential system of facets. 

\begin{lem}
\label{lemreal}
Let $X=(\mathbb{C}^n, \| \cdot \|)$ be either a complex polytope space with a real essential system of vertices or an adjoint complex polytope space with a real essential system of facets. Then for any two vectors $x, y \in \mathbb{R}^n$ we have $\|x+iy\| \geq \|x\|$. 
\end{lem}
\begin{proof}
Let us first assume that $\| \cdot \|$ is a complex polytope norm with an essential system of vertices $u_1, \ldots, u_N \in \mathbb{R}^n$. Then for every $z \in \mathbb{C}^n$ we have
$$\|z\| = \min \left \{ \sum_{j=1}^{N} |\lambda_j|: \ z = \sum_{j=1}^{N} \lambda_j u_j \right \}. $$
Let us write $z = x+iy$ for $x, y \in \mathbb{R}^n$ and take $\lambda_1, \ldots, \lambda_N \in \mathbb{C}$ for which the minimum above is attained. If we write $\lambda_j = \alpha_j + i\beta_j$ for $\alpha_j, \beta_j \in \mathbb{R}$ then we see that $x = \sum_{j=1}^{N} \alpha_j u_j$, as the vectors $u_1, \ldots, u_N$ are real. Hence
$$\|x\| \leq \sum_{j=1}^{N} |\alpha_j| \leq \sum_{j=1}^{N} \sqrt{\alpha^2_j+\beta^2_j} = \sum_{j=1}^{N} |\lambda_j|=\|x+iy\|$$
and the conclusion follows.

Now let us suppose that $\| \cdot \|$ is an adjoint complex polytope norm with a real essential system of facets $f_1, \ldots, f_N: \mathbb{C}^n \to \mathbb{C}$. In this case for every $z \in \mathbb{C}^n$ we have
$$\|z\| = \max \{|f_j(z)|: 1 \leq j \leq N\}.$$
Let us write $z = x+iy$ for $x, y \in \mathbb{R}^n$ and take $1 \leq j \leq N$ such that $\|x\|=|f_j(x)|$. Since by assuption $f_j$ is a real functional we have
$$\|x\| = |f_j(x)| \leq \sqrt{f_j(x)^2 + f_j(y)^2} = |f_j(x) + if_j(y)| = |f_j(z)|  \leq \|z\| = \|x+iy\|$$
and the proof is finished in both situations.
\end{proof}

In the following result we collect some information about best approximation in real subspaces of complex polytope spaces with a real essential system of vertices and of adjoint complex polytope spaces with a real essential system of facets.

\begin{twr}
\label{twrpolytope}
Let $X=(\mathbb{C}^n, \| \cdot \|)$ be either a complex polytope space with a real essential system of vertices or an adjoint complex polytope space with a real essential system of facets. Suppose that $Y \subseteq X$ is a real subspace and $x \in X \setminus Y$ is a real vector, for which $y_0 \in Y$ is a unique best approximation in $Y$. Then $y_0$ is a real vector and it is a $2$-strongly unique best approximation for $x$. Additionally, if $X$ is an adjoint complex polytope space, then $y_0$ is not $\alpha$-strongly unique best approximation for any $\alpha<2$.
\end{twr}
\begin{proof}
Since $Y$ is a real subspace we have $\re(y_0), \im(y_0) \in Y$ and Lemma \ref{lemreal} implies that
$$\|x-y_0\|=\|(x-\re(y_0)) + i\im(y_0)\| \geq \|x-\re(y_0)\|,$$
which shows that the distance of $\re(y_0) \in Y$ to $x$ is not greater than of $y_0$. Since $y_0$ is a unique best approximation for $x$ it follows that $y_0=\re(y_0)$ is a real vector.

Now we shall consider the two situations separately and we first start with the setting of $\| \cdot \|$ being an adjoint complex polytope norm. The fact that $y_0$ is a $2$-strongly unique best approximation for $x$ follows directly from Theorem \ref{twr2strgeneral} applied to a vector $x-y_0 \in \mathbb{R}^n$. To have the setting of adjoint complex polytope norms done, we are left with proving that $y_0$ is not an $\alpha$-strongly unique best approximation for any $\alpha<2$. Indeed, let us assume the contrary and let $f_1, \ldots, f_N: \mathbb{C}^n \to \mathbb{C}$ be the real essential system of facets, i.e.
$$\|z\| = \max \{|f_j(z)|: \ 1 \leq j \leq N\}$$
for $z \in \mathbb{C}^n$. By considering $x-y_0$ instead of $x$ we can suppose that $y_0=0$ and by rescalling we can further assume that $\|x\|=1$. Now let $y \in Y$ be any fixed non-zero vector with real coordinates. Then by the assumed $\alpha$-strong uniqueness for every $t > 0$ we have
$$\|x-tiy\|^{\alpha} \geq 1 + rt^{\alpha},$$
for some constant $r>0$. We note that if $1 \leq j \leq N$ is such that $|f_j(x)|<1$ then for $t>0$ small enough we have $|f_j(x-tiy)|<1$. Hence, for $t>0$ sufficiently small, the norm $\|x-tiy\|$ must be realized for an index $1 \leq j \leq N$ such that $|f_j(x)|=1$. By taking $t=\frac{1}{k}$ for positive integer $k$ and letting $k \to \infty$, we see that some index $1 \leq j \leq N$ will realize the norm $\|x-tiy\|$ infinitely many times. In other words, there exists an index $1 \leq j \leq N$ such that $|f_j(x)|=1$ and for infinitely many $k \geq 1$ we have
$$\left |f_j\left ( x-\frac{i}{k}y \right ) \right |^{\alpha} = \left | 1 - \frac{i}{k}f_j(y) \right |^{\alpha} \geq 1 + \frac{r}{k^{\alpha}}.$$
Obviously we must have $a=f_j(y) \neq 0$. This can be rewritten as
$$k^{\alpha}\left ( \left ( 1 + \frac{a^2}{k^2} \right )^{\frac{\alpha}{2}} - 1) \right ) \geq r.$$
However, by the L'Hospital's rule and the assumption $\alpha<2$ we have
$$\lim_{t \to 0^+} \frac{\left ( 1 + t^2a^2 \right )^{\frac{\alpha}{2}} - 1}{t^{\alpha}} = \lim_{t \to 0^+} \frac{a^2(\left ( 1 + t^2a^2 \right )^{\frac{\alpha}{2}-1}}{t^{\alpha-2}} = \lim_{t \to 0^+} a^2 \frac{t^{2-\alpha}}{\left ( 1 + t^2a^2 \right )^{1-\frac{\alpha}{2}}} =0.$$
Hence, for $t=\frac{1}{k}$ with $k$ sufficiently large we get a contradiction and the conclusion follows.

We can now turn our attention to the case of $\| \cdot \|$ being a complex polytope norm. Let $u_1, \ldots, u_N \in \mathbb{R}^n$ be an essential system of vertices of $\| \cdot \|$, so that for $z \in \mathbb{C}^n$ we have
$$\|z\| = \min \left \{ \sum_{j=1}^{N} |\lambda_j|: \ z = \sum_{j=1}^{N} \lambda_j u_j \right \}.$$
As before we can assume that $y_0=0$ and $\|x_0\|=1$.

We shall prove that for some constant $r>0$ we have
\begin{equation}
\label{ineqteza}
\|x-y\| \geq 1 + r\|y\|^2
\end{equation}
for every $y \in Y$, as squaring this inequality yields the desired $2$-strong uniqueness. Moreover, by Lemma \ref{lemcompact} it is enough to establish \eqref{ineqteza} only for $\|y\| \leq 1$.

Because the norm $ \| \cdot \|$, when considered as a norm on $\mathbb{R}^n$, is a polytope norm (with the unit ball $\conv \{\pm u_1, \ldots, \pm u_N\}$) and $Y$ is a real subspace, we know that $0$ is a $1$-strongly unique best approximation for $x$ in a subspace of real parts of vectors from $Y$. In other words, there exists a constant $r_1>0$ such that for every $y \in Y$ we have
$$\|x-\re(y)\| \geq 1 + r_1\|\re(y)\|.$$
In particular, if $y \in Y$ is such a vector that
$$\|\re(y)\| \geq C \|y\|^2$$
for some constant $C \in (0, 1)$ to be specified later, then \eqref{ineqteza} holds with a constant $Cr_1$. Thus, let us assume that the opposite inequality is true. In particular
$$\|\re(y)\| \leq C \|y|^2 \leq C \|y\|.$$
However, on the other hand we have
$$\|y\| = \|\re(y) + i\im(y)\| \leq \|\re(y)\| + \|\im(y)\|$$
and hence
\begin{equation}
\label{ineqim}
\|\im(y)\| \geq (1-C)\|y\|.
\end{equation}
We can now estimate
$$\|x-y\| = \|x - \re(y) - i\im(y)\| \geq \|x - i \im(y)\| - \|\re(y)\| \geq \|x - i \im(y)\| - C\|y\|^2.$$
If we now prove that $\|x - i \im(y)\| \geq 1 + C'\|y\|^2$ for some constant $C'>C$ then the inequality \eqref{ineqteza} will follow with the constant $C'-C>0$. Let us now take a representation
$$x - i \im(y) = (a_1 - ib_1)u_1 + \ldots + (a_N - ib_N)u_N$$
such that $\|x - i \im(y)\| = \sum_{j=1}^{N} |a_j - ib_j| = \sum_{j=1}^{N} \sqrt{a_j^2 + b_j^2}$. Then as $x, \im(y)$ and $u_1, \ldots, u_N$ are all real vectors we have $x = \sum_{j=1}^{N} a_ju_j$ and $\im(y) = \sum_{j=1}^{N} b_ju_j$ so in particular 
\begin{equation}
\label{oszacowania}
\sum_{j=1}^{n} |a_j| \geq \|x\|=1 \quad \text{  and  } \quad \sum_{j=1}^{n} |b_j| \geq \|\im(y)\|.
\end{equation}
Moreover, we note that since $\|x-i \im(y)\| \leq \|x\| + \|\im(y)\| \leq 1 + \|y\| \leq 2$ (using Lemma \ref{lemreal}) we have $|a_j|, |b_j| \leq 2$ for every $1 \leq j \leq N$. Let us now observe that for any two real numbers $a$, $b$ with $|a|, |b| \leq 2$ the following inequality is true
$$\sqrt{a^2 + b^2} \geq |a| + \frac{1}{8}b^2.$$
Indeed, after squaring this is equivalent to
$$b^2 \geq \frac{1}{4}|a|b^2 + \frac{1}{64}b^4.$$
However, from $|a|, |b| \leq 2$ we get
$$\frac{1}{4}|a|b^2 + \frac{1}{64}b^4 \leq \frac{1}{2}b^2 + \frac{1}{16}b^2 = \frac{9}{16} b^2 \leq b^2.$$
Since $|a_j|, |b_j| \leq 2$ we can use the above estimate, combined with the the Cauchy-Schwarz inequality, and estimates \eqref{ineqim},  \eqref{oszacowania} to obtain
$$\|x - i\im(y)\| = \sum_{j=1}^{N} \sqrt{a_j^2 + b_j^2} \geq \sum_{j=1}^{N} \left ( |a_j| + \frac{1}{8} |b_j|^2 \right ) \geq 1 + \frac{1}{8} \sum_{j=1}^{N} |b_j|^2$$
$$ \geq 1 + \frac{1}{8N} \left ( \sum_{j=1}^{N} |b_j| \right )^2 \geq 1 + \frac{1}{8N} \|\im(y)\|^2 \geq 1 + \frac{1-C}{8N}\|y\|^2.$$
Thus, we can take $C' = \frac{1-C}{8N}$ and it is enough to choose $C>0$ so that $\frac{1-C}{8N} > C$. In particular, if we take $C=\frac{1}{16N+1}$ then $C'=2C$. Altogether, the desired inequality \eqref{ineqteza} holds for $\|y\| \leq 1$ with a constant $ r = \frac{r_1}{16N+1}>0$.
\end{proof}

The last part of the previous result highlights a notable contrast in approximation properties between complex polytope spaces and adjoint complex polytope spaces. We saw that in a setting of the previous result, a $2$-strong uniqueness can never be improved to an $\alpha$-strong uniqueness for any constant $\alpha<2$ in an adjoint complex polytope space. However, the situation is vastly different for complex polytope spaces, where in fact the best possible $1$-strong uniqueness can hold, even for every real subspace. This happens again in the classical space $\ell_1^n$, for which we give the following strengthening of the previous result.

\begin{twr}
\label{twrl12}
Let $Y \subseteq \ell_{1}^n$ be a real subspace of the complex $\ell_1^n$ space. Suppose that $y_0 \in Y$ is a unique best approximation in $Y$ for a real vector $x \in \ell_{1}^n \setminus Y$. Then $y_0$ is a real vector and it is a $1$-strongly unique best approximation for $x$ in $Y$.
\end{twr}
\begin{proof}
The fact that $y_0$ is a real vector follows from the previous theorem. By considering $x-y_0$ instead of $x$ we can assume that $y_0=0$. We can further suppose that $x_j \neq 0$ for $1 \leq j \leq m$, while $x_j=0$ for $m+1 \leq j \leq n$ (possibly $m=n$). By a characterization given in Lemma \ref{lemchar} for every $y \in Y$ we have
\begin{equation}
\label{condy}
\left | \sum_{j=1}^{m} \frac{x_j}{|x_j|} y_j \right | \leq \sum_{j=m+1}^{n} |y_j|.
\end{equation}
Let $\mathcal{S} \subseteq S_Y$ be defined as
$$\mathcal{S} = \left \{ y \in S_Y: \  \left | \sum_{j=1}^{m} \frac{x_j}{|x_j|} y_j \right | = \sum_{j=m+1}^{n} |y_j| \right \}.$$
Arguing in exactly the same way as in proof of Theorem \ref{twrl1} we can prove that if the set $\mathcal{S}$ is empty, then $0$ is a $1$-strongly unique best approximation for $x$. Thus, let us assume that $y \in \mathcal{S}$. Since $ty \in \mathcal{S}$ for any $t \in \mathbb{C}$ with $|t|=1$, we can further assume that
$$\sum_{j=1}^{m} \frac{x_j}{|x_j|} y_j  = \sum_{j=m+1}^{n} |y_j|.$$
In particular $\im \left ( \sum_{j=1}^{m} \frac{x_j}{|x_j|} y_j \right ) = 0$. Since $Y$ is a real subspace, we have $\re(y) \in Y$ and by applying \eqref{condy} to this vector we obtain
$$ \sum_{j=1}^{m} \frac{x_j}{|x_j|} y_j = \left | \re \left ( \sum_{j=1}^{m} \frac{x_j}{|x_j|} y_j  \right ) \right | = \left | \sum_{j=1}^{m} \frac{x_j}{|x_j|} \re(y_j) \right | \leq \sum_{j=m+1}^{n} |\re(y_j)| \leq \sum_{j=m+1}^{n} |y_j|.$$
Hence, in every estimate we must have an equality. In particular, assuming that $\re(y) \neq 0$, we have proved that $\frac{\re(y)}{\|\re(y)\|_1} \in \mathcal{S}$. However, this violates a condition for uniqueness given in Lemma \ref{lemchar} as $\im(x_j\re(y_j)) = 0$ for any $1 \leq j \leq m$. If $\re(y)=0$ then $\im(y_j) \neq 0$ for some $1 \leq j \leq m$ (as $y \neq 0$) and the equality in the estimates above implies that $y_j=0$ for $m+1 \leq j \leq n$. Thus, in this case $\frac{\re(-iy)}{\|\re(iy)\|_1}=\frac{\im(y)}{\|\im(y)\|_1}$ is a real non-zero vector belonging to $\mathcal{S}$ and again we obtain a contradiction with the uniqueness of $0$ as a best approximation for $x$. This finishes the proof.
\end{proof}

We emphasize that we do not know if the strengthening of Theorem \ref{twrpolytope} to the $1$-strong uniqueness is possible for all complex polytope norms or is it something particular to the $\ell_1^n$ space. The situation is similar to that of Theorem \ref{twrl1} where we have established a condition for the space $\ell_1^n$, which we were not able to get for all complex polytope norms. It is not clear, if the space $\ell_1^n$ is representative for all complex polytope norms or if it simply has some exceptionally good approximation properties.

\section{$2$-strong uniqueness of minimal projections}
\label{sectionproj}

In this section we use previous results and similar methods to establish $2$-strong uniqueness of some minimal projections. Our basic observation is the fact that if $X$ is a finite-dimensional normed space and $Y \subseteq X$ is a subspace, then the problem of determining if a given projection $P: X \to Y$ is minimal, can be equivalently stated as a problem of finding a best approximation in the space of linear operators. To be more precise, a projection $P$ is a minimal if and only if, the zero operator is a best approximation for $P$ (when the standard operator norm is considered) in the following subspace of the space $\mathcal{L}(X, Y)$ of all continuous linear operators from $X$ to $Y$:
$$\mathcal{L}_Y(X, Y) = \{L \in \mathcal{L}(X, Y): \ L|_Y \equiv 0\}. $$
Moreover, $P$ is a unique minimal projection if and only if, $0$ is a unique best approximation for $P$ in $\mathcal{L}_Y(X, Y)$, and similarly, $P$ is an $\alpha$-strongly unique minimal projection if and only if, $0$ is an $\alpha$-strongly unique best approximation for $P$. This observation allows us to approach problems of minimal projections with the methods of approximation theory.

Before going to the setting of projections, we first need to recall one more general result. The following lemma of Smarzewski \cite{smarzewski} is in itself a very useful tool for establishing a $2$-strong uniqueness of an element of a best approximation. Let us recall that by Lemma \ref{lemfunkcjonal}, the zero vector is a best approximation for $x$ in $Y$ if and only if we can find a linear functional $f$ such that $f(x)=\|f\|=1$ and $f|_Y \equiv 0$. Lemma of Smarzewski states that if such a functional can be written as a convex combination of functionals of norm one, that are total over $Y$, then the best approximation is $2$-strongly unique. For the reader's convenience, we include a short proof.

\begin{lem}
\label{lem2strong}
Let $X$ be a normed space over $\mathbb{K}$, let $Y \subseteq X$ be a finite-dimensional subspace and let $x \in X \setminus Y$ be a vector. Suppose that $f_1, \ldots, f_l \in S_{X^*}$ are linear functionals and $\alpha_1, \ldots, \alpha_l$ are positive reals with the sum $1$, such that a functional $f \in X^*$ defined as
$$f =\sum_{j=1}^{l} \alpha_jf_j$$
satisfies $f(x)=\|x\|$ and $f|_Y \equiv 0$. Assume additionally that there does not exist a non-zero vector $y_0 \in Y$ such that $f_j(y_0)=0$ for every $1 \leq j \leq l$. Then $0$ is a $2$-strongly unique approximation for $x$.
\end{lem}
\begin{proof}
Since $f(x)=\|x\|$ and $f$ is a convex combination of the functionals $f_j$, the triangle inequality readily implies that $f_j(x)=\|x\|$ for every $1 \leq j \leq l$. Moreover, from the assumption it follows that the function $\| \cdot \|_0$ given on $Y$ as
$$\|y\|_0 = \left ( \sum_{j=1}^{l} \alpha_j |f_j(y)|^2 \right )^{\frac{1}{2}},$$
is a norm on $Y$. Because the space $Y$ is finite-dimensional, there exists a constant $r>0$ such that $\|y\|_0 \geq r\|y\|$ for every $y \in Y$. For any fixed $y \in Y$ we can now estimate
$$\|x-y\|^2 \geq |f_j(x-y)|^2= |f_j(x) - f_j(y)|^2$$
$$=|f_j(x)|^2 - 2\re(f_j(x) \cdot f_j(y)) + |f_j(y)|^2$$
$$=\|x\|^2 - 2\|x\|\re(f_j(y)) + |f_j(y)|^2.$$
Multiplying this by $\alpha_j>0$ and summing over all $1 \leq j \leq l$ we get
$$\|x-y\|^2 \geq \|x\|^2 - 2 \|x\| \re \left ( \sum_{j=1}^{l} \alpha_j f_j(y) \right ) + \sum_{j=1}^{l} \alpha_j |f_j(y)|^2$$
$$=\|x\|^2 + \|y\|_0^2 \geq \|x\|^2 + r^2\|y\|^2,$$
which proves that $0$ is a $2$-strongly unique best approximation for $x$ with a constant $r^2>0$.

\end{proof}

To be fully able to apply methods of approximation theory to problems of minimal projections, we shall need a notion of the Chalmers-Metcalf functional/operator. It is a powerful tool with many applications, but in the following we shall provide only a concise introduction to this concept along with some basic properties. As we have already noted, a projection $P: X \to Y$ is minimal if and only if, the zero operator is a best approximation for $P$ in the subspace $\mathcal{L}_Y(X, Y)$. If Lemma \ref{lemfunkcjonal} is applied in the space of linear operators $\mathcal{L}(X, Y)$, then a functional from this lemma is a so called \emph{Chalmers-Metcalf functional} $F \in S_{L^*(X, Y)}$ for $Y$. For a sake of simplicity, let us assume that $X$ is a finite-dimensional, as we will be concerned only with such a situation. Using the well-known characterization of extreme points of the unit ball of the space $\mathcal{L}^*(X, Y)$, this functional $F$ of norm $1$ can be then written in the form
$$F = \sum_{j=1}^{l} \alpha_j x_j \otimes f_j,$$
where $\alpha_j$'s are positive reals with sum $1$ and $(x_j, f_j) \in S_{X} \times S_{Y^*}$ for $1 \leq j \leq l$. In this case by $x \otimes f \in \mathcal{L}^*(X, Y)$ (where $x \in X$ and $f \in Y^*$) we understand a functional defined as $(x \otimes f)(L) = f(L(x))$ for $L \in \mathcal{L}(X, Y)$. Considering the fact that, as in Lemma \ref{lemfunkcjonal}, a functional $F$ satisfies $F(P)=\|P\|=\lambda(Y, X)$, it can be easily deduced that for every $1 \leq j \leq l$ we have $f_j(P(x_j))=\|P\|=\lambda(Y, X)$, i.e. a pair $(x_j, f_j) \in S_{X} \times S_{Y^*}$ is a \emph{norming pair} for the projection $P$. Moreover, the condition $f|_Y \equiv 0$ from Lemma \ref{lemfunkcjonal} translates now to $F(L)=0$ for every $L \in \mathcal{L}_Y(X, Y)$.

Using trace duality, one can obtain an equivalent operator version of the Chalmers-Metcalf functional. For $x \in X$ and $f \in Y^*$, we defined $x \otimes f$ as a functional on $\mathcal{L}(X, Y)$, but this can be regarded also as a rank one operator from $Y$ to $X$ defined as $Y \ni y \to f(y)x \in X$. By a \emph{Chalmers-Metcalf operator} $T: Y \to X$ we shall understand an operator of the form:
$$T = \sum_{j=1}^{l} \alpha_j x_j \otimes f_j,$$
where again $(x_j, f_j) \in S_{X} \times S_{Y^*}$ is a norming pair for $P$ for every $1 \leq j \leq l$ and operator $T$ satisfies $T(Y) \subseteq Y$ (which is a consequence of the condition $F|_{\mathcal{L}_Y(X, Y)} \equiv 0$ in the functional version). In this case we have a simple formula for the projection constant given as $\lambda(Y, X) = \tr(T)$ (the trace of $T$ is properly defined as $T(Y) \subseteq Y$).

To summarize this short discussion, the problem of determining if a projection $P$ is minimal is equivalent to verifying if some convex combination of its norming pairs $x_j \otimes f_j$ gives rise to a Chalmers-Metcalf functional/operator. In the case of functional this convex combination has to vanish on the subspace $\mathcal{L}_Y(X, Y)$, while in the operator version $Y$ has to be an invariant subspace of the Chalmers-Metcalf operator. It turns out that a Chalmers-Metcalf functional/operator constructed for one minimal projection $P$ works equally well for every other minimal projection, so that the norming pairs $(x_j, f_j)$ appearing in the definition, turn out to be norming pairs for every minimal projection. In general a Chalmers-Metcalf functional/operator does not have to be unique. We refer the reader to \cite{lewickiskrzypekoperator2} for a thorough introduction to the notion of Chalmers-Metcalf operator and an extensive discussion of its properties.

In \cite{lewickiskrzypekoperator} a connection between a uniqueness of minimal projection and the Chalmers-Metcalf operator was studied. In particular, in Corollary $2.3$ it was proved that under an assumption of smoothness of $X$, the invertibility of at least one Chalmers-Metcalf operator for $Y$, implies that a minimal projection $P: X \to Y$ is unique. Below we observe that in the hyperplane case the assumption of smoothness can in fact be dropped, assuming that a hyperplane is non $1$-complemented. Moreover, the uniqueness can be improved to the $2$-strong uniqueness. From the proof it is easy to see that the $2$-strong uniqueness holds true also for the non-hyperplane setting under the additional smoothness assumption (as in Corollary $2.3$ from \cite{lewickiskrzypekoperator}).

\begin{lem}
\label{lemproj}
Let $X$ be an $n$-dimensional normed space over $\mathbb{K}$ and let $Y \subseteq X$ be a subspace of dimension $n-1$ such that $\lambda(Y, X)>1$. Assume that $T:Y \to X$ given by
$$T=\alpha_1 x_1 \otimes f_1 +  \ldots + \alpha_l x_l \otimes f_l,$$
is a Chalmers-Metcalf operator for $Y$. If there does not exist a non-zero vector $y_0 \in Y$ such that $f_j(y_0)=0$ for all $1 \leq j \leq l$, then a minimal projection $P: X \to Y$ is $2$-strongly unique. In particular, a minimal projection is $2$-strongly unique when $T$ is an injective mapping.
\end{lem}

\emph{Proof.} By Lemma \ref{lem2strong} it is enough to prove that if an operator $L \in \mathcal{L}_Y(X, Y)$ satisfies $f_j(L(x_j))=0$ for every $1 \leq j \leq l$, then $L \equiv 0$. Indeed, because $Y$ is of the dimension $n-1$, the operator $L$ is of the rank at most $1$ and hence the image of $L$ is of the form $\lin\{y_0\}$ for some $y_0 \in Y$. We note that since $\lambda(Y, X)>1$ we have $x_j \not \in Y$ for every $1 \leq j \leq l$, as the norming points of a minimal projection onto $Y$ can not belong to $Y$ in this case. In particular, if $L(x_j)=0$ for at least one index $1 \leq j \leq l$, then $L \equiv 0$. Assuming otherwise, we have $L(x_j)=a_jy_0$ for some $a_j \neq 0$ and hence if $f_j(L(x_j))=0$, then also $f_j(y_0)=0$. Because this holds for every $1 \leq j \leq l$ we must have $y_0=0$ by the assumption and the conclusion follows. \qed

A well-known theorem of Odyniec \cite{odinec2} says that if $Y$ is a two-dimensional subspace of a three-dimensional real normed space $X$ such that $\lambda(Y, X)>1$, then the minimal projection onto $X$ is unique. It was proved later in \cite{lewicki} (Theorem $2.1$) that, in this case, the minimal projection is actually even $1$-strongly unique. Although the theorem of Odyniec is known already from $1978$, to our best knowledge, the complex counterpart was never studied. In the following theorem we extend theorem of Odyniec to the complex setting, proving that in this case the minimal projection is not only unique, but even $2$-strongly unique. It should be emphasized that, while in this paper we usually assume that uniqueness holds a priori, this is not the case this time. The proof works equally well for real and complex scalars, although the result is new only in the complex case.

\begin{twr}
\label{twrcompl}
Let $X = (\mathbb{K}^n, \| \cdot \|)$ be a three-dimensional normed space over $\mathbb{K}$ and let $Y \subseteq X$ be a two-dimensional subspace. If $\lambda(Y, X)>1$, then a minimal projection $P: X \to Y$ is $2$-strongly unique.
\end{twr}
\begin{proof}. Let $T:Y \to X$ given by
$$T=\alpha_1 x_1 \otimes f_1 + \ldots + \alpha_l x_l \otimes f_l,$$
be a Chalmers-Metcalf operator for $Y$. By Lemma \ref{lemproj} it is enough to prove that if $y_0 \in Y$ satisfies $f_j(y_0)=0$ for every $1 \leq j \leq l$, then $y_0=0$. For the sake of a contradiction, let us suppose that $y_0 \neq 0$. Since $Y^*$ is a two-dimensional space, this implies the functionals $f_1, \ldots, f_l \in Y^*$ span a subspace of dimension $1$ in $Y^*$, that is, every two of them are linearly dependent. Because they are of norm $1$, for every $2 \leq j \leq l$ there exists $a_j \in \mathbb{K}$, $|a_j|=1$ such that $f_j = a_j f_1$. Thus, for every $y \in Y$ we have
$$T(y) = f_1(y) \left ( \alpha_1 x_1  + \alpha_2 a_2 x_2 + \ldots + \alpha_l a_l x_l \right ).$$
Clearly, $f_1$ is a non-zero functional, so we can take $y \in Y$ such that $f_1(y)=1$. From this and the fact that $Y$ is an invariant subspace for $T$, we get that
$$\alpha_1 x_1 + \alpha_2 a_2 x_2 + \ldots + \alpha_l a_l x_l \in Y.$$
Therefore, we can calculate the trace of of $T$ simply as
$$\tr\left (T \right ) = f_1 \left ( \alpha_1 x_1 + \alpha_2 a_2 x_2 + \ldots + \alpha_l a_l x_l \right ).$$
However, as $\|x_j\|=1$ for every $1 \leq j \leq l$, by the triangle inequality we have that
$$\|\alpha_1 x_1 + \alpha_2 a_2 x_2 + \ldots + \alpha_l a_l x_l\| \leq 1.$$
Hence
$$1<\lambda(Y, X)=|\tr\left (T \right )| = \left | f_1 \left ( \alpha_1 x_1 + \alpha_2 a_2 x_2 + \ldots + \alpha_l a_l x_l \right ) \right |$$
$$\leq \|f_1\| \cdot |\alpha_1 x_1 + \alpha_2 a_2 x_2 + \ldots + \alpha_l a_l x_l\| \leq 1,$$
which leads to a contradiction. This proves that $y_0=0$ and in consequence a minimal projection from $X$ onto $Y$ is $2$-strongly unique. 
\end{proof}

In both real and complex cases, already the uniqueness from the previous result fails without the assumption $\lambda(Y, X)>1$. Moreover, there is also no uniqueness for hyperplanes of general $n$-dimensional normed spaces, when $n \geq 4$. For examples, see Remarks $2.2$ and $2.3$ in \cite{lewicki}. As we shall see soon, a $2$-strong uniqueness can not be improved for the non $1$-complemented hyperplanes of the complex space $\ell_{\infty}^3$ (Theorem \ref{twrlinfty}), so the previous result is best possible in the complex setting. In the example below we consider a setting of projections onto $1$-dimensional subspaces. It turns out, that in this situation, even if the minimal projection is unique, it does not have to be $2$-strongly unique. 

\begin{example}
\label{example1dim}
Let $X$ be an $n$-dimensional normed space over $\mathbb{K}$ (where $n \geq 2$) and let $Y = \lin\{y\}$, where $y \in X$ is a unit vector. Let $f_0 \in Y^*$ be a functional such that $\|f_0\|=f_0(y)=1$. Then a minimal projection $P_{f_0}: X \to Y$ with $\|P_{f_0}\|=1$ is given by a formula $P_{f_0}(x)= f_0(x)y$. Clearly, if such a functional $f_0$ is unique (which happens, for example, when $X$ is a smooth space), then a minimal projection onto $Y$ is also unique. Any other projection onto $Y$ can be written as $P_{f}(x)=\frac{f(x)}{f(y)} y$ where $f \in Y^*$ is any functional such that $f(y) \neq 0$. Then $\|P_f\| = \frac{\|f\|}{|f(y)|}$ and $\|P_f - P_{f_0}\| = \frac{1}{|f(y)|}\|f - f(y)f_0\|$. Thus, the condition for an $\alpha$-strong uniqueness of $P_{f_0}$
$$\|P_{f}\|^{\alpha} \geq \|P_{f_0}\|^{\alpha} + r \|P_f - P_{f_0}\|^{\alpha}$$
rewrites as
$$\|f\|^{\alpha} \geq |f(y)|^{\alpha} + r\|f - f(y)f_0\|^{\alpha}.$$
Let us now take a concrete example $X = \ell_p^n$, where $1 < p < 2$ and $y=(1, 0, \ldots, 0)$. Then $X$ is a smooth space and $f_0(x)=x_1$ for $x \in \ell_p^n$. Moreover, let us take $f$ given by a vector $(1, t, 0, \ldots, 0)$ where $t>0$. Then $\|f\|=\left ( 1 + t^q \right )^{\frac{1}{q}}$, where $q = \frac{p}{p-1}>2$. Similarly we calculate $\|f-f(y)f_0\| = t$. Hence, the condition for an $\alpha$-strong uniqueness can be rewritten as
$$\left ( 1 + t^q \right )^{\frac{\alpha}{q}} \geq 1 + rt^{\alpha}.$$
However, by the L'Hospital's rule we have
$$\lim_{t \to 0^+} \frac{\left ( 1 + t^q \right )^{\frac{\alpha}{q}} - 1}{t^{\alpha}} = \lim_{t \to 0^+} \left ( 1 + t^q \right )^{\frac{\alpha}{q}-1} \cdot t^{q-\alpha}.$$
We see that if $\alpha<q$, then the above limit is equal to $0$ and the required condition fails. In other words, an $\alpha$-strong uniqueness can hold only for $\alpha \geq q$. By taking $p$ arbitrarily close to $1$, it is clear that $q$ can be arbitrarily large. In other words, there is no a single $\alpha$ which would guarantee the $\alpha$-strong uniqueness of a unique minimal projection for all normed spaces of a given dimension and their $1$-dimensional subspaces.
\end{example}

In the next theorem we move to a particular setting of hyperplanes of the complex $\ell_{\infty}^n$ space, where we prove that for every hyperplane with a unique minimal projection, this projection is automatically $2$-strongly unique. In the real setting, the $1$-strong uniqueness follows immediately from Theorem \ref{twrsudolskiwojcik}, as in this case, a space of linear operators $\mathcal{L}^* (\ell_{\infty}^n, Y)$ is easily seen to be a real polytope space. Furthermore, it is known that in the case of the $1$-complemented hyperplanes a unique minimal projection is $1$-strongly unique also in the complex setting (see Theorem $2.3.1$ (a) in \cite{lewickihabilitacja}). On the other hand, as we shall see, for the non $1$-complemented hyperplanes, the $2$-strong uniqueness can be never improved in the complex case. We recall that conditions characterizing uniqueness of minimal projection onto a hyperplane in $\ell_{\infty}^n$ are known. If $Y \subseteq \ell_{\infty}^n$ is a hyperplane given as $Y=\ker f$, where $f=(f_1, \ldots, f_n) \in \mathbb{C}^n$ with $\|f\|_1=1$, a minimal projection onto $Y$ has norm greater than $1$ and is unique if and only if $0 < |f_j| < \frac{1}{2}$ for every $1 \leq j \leq n$ (see Theorem II.3.6 (b) in \cite{odynieclewicki} and Section 2 in \cite{skrzypekoperator}). Therefore, only such a situation is considered below.

\begin{twr}
\label{twrlinfty}
Let $Y \subseteq \ell_{\infty}^n$ be a subspace given as $Y=\ker f$, where a vector $f=(f_1, \ldots, f_n) \in \mathbb{C}^n$ satisfies conditions: $\|f\|_1=1$ and $0 < |f_j| < \frac{1}{2}$ for every $1 \leq j \leq n$. Then the minimal projection from $\ell_{\infty}^n$ onto $Y$ is $2$-strongly unique, but it is not $\alpha$-strongly unique for any $\alpha<2$.
\end{twr}

\begin{proof}
 We begin with observing that since a mapping of the form
$$\ell_{\infty}^n \ni (x_1, x_2, \ldots, x_n) \to (a_1x_1, a_2x_2, \ldots, a_nx_n) \in \ell_{\infty}^n$$
is an isometry for any complex scalars $a_i$ of modulus $1$, we can suppose throughout the proof that $f_j$ is a real positive number for any $1 \leq j \leq n$.

Let $P: \ell_{\infty}^n \to Y$ be a minimal projection. Because $(\ell_{\infty}^n)^*=\ell^n_1$, a Chalmers-Metcalf operator $T: \ell_{\infty}^n \to \ell_{\infty}^n$ for $Y$ can be written in a form
$$T(z) = \alpha_1 e_1^*(z) x_1 + \alpha_2 e_2^*(z) x_2 \ldots + \alpha_n e_n^*(z) x_n,$$
where $e_j^*$ is a linear functional given as $e_j^*(z)=z_j$ for $z \in \ell_{\infty}^n$ and $x_j \in S_{\ell_{\infty}^n}$. Moreover, it was proved in \cite{skrzypekoperator} (Theorem $2$) that in our situation we have $\alpha_j>0$ for every $1 \leq j \leq n$ (it should be noted that in \cite{skrzypekoperator} only the real case was formally considered, but the proof of Theorem 2 presented there works equally well also in the complex case). Clearly, if $z \in \ell_{\infty}^n$ is a vector satisfying $e_i^*(z)=0$ for every $1 \leq j \leq n$, then $z=0$. Thus, the $2$-strong uniqueness of $P$ follows now immediately from Lemma \ref{lemproj}. 

Now we shall prove that $P$ is not $\alpha$-strongly unique for any $\alpha < 2$. Let us assume otherwise. Then, there exists a constant $r>0$, such that for any projection $Q: \ell_{\infty}^n \to Y$ we have
\begin{equation}
\label{stala}
\|Q\|^{\alpha} \geq \|P\|^{\alpha} + r\|Q-P\|^{\alpha}.
\end{equation}
A general projection $Q:\ell_{\infty}^n \to Y$ is of the form $Q(z)=z-f(z)v$, where $v \in \ell_{\infty}^n$ is a vector such that $f(v)=1$. For any vector $z \in \mathbb{C}^n$ with $\|z\|_{\infty} \leq 1$ we have
$$|Q(z)_j| = |z_j-f(z)v_j| = \left | -\left ( \sum_{k \neq j} f_kz_k \right )v_i + (1-f_jv_j)z_j  \right |$$
$$\leq \sum_{k \neq j} f_k|v_j| + |1-f_jv_j| = (1-f_j)|v_j| + |1-f_jv_j|.$$
Moreover, the equality can be attained for a suitable $z \in \mathbb{C}^n$ satisfying $\|z\|_{\infty}=1$. Hence
\begin{equation}
\label{norma}
\|Q\| = \max_{1 \leq j \leq n} \left ( (1-f_j)|v_j| + |1-f_jv_j| \right ).
\end{equation}
Let us now suppose that the minimal projection $P$ is given as $P(z)=z-f(z)w$, where $w \in \mathbb{C}^n$ satisfies $f(w)=1$. It is well-known (see for example Theorem $4$ in \cite{skrzypekoperator}) that $\|P\| = 1 + \lambda,$ where 
$$\lambda =\left ( \sum_{j=1}^{n} \frac{f_j}{1-2f_j} \right )^{-1}$$
and the unique vector $w$ is given as $w_j=\frac{\lambda}{1-2f_j}$ for $1 \leq j \leq n$. In particular, $w_j$ is a positive real number. Moreover, the following property will be crucial for us: the maximum in (\ref{norma}) giving the norm of the projection $P$ is attained for every coordinate $1 \leq j \leq n$, that is
$$\|P\| = 1 + \lambda = (1-f_j)|w_j| + |1-f_jw_j|$$
for every $1 \leq j \leq n$. This can be verified simply by a direct calculation, but it is also a consequence of the fact that we have $\alpha_j>0$ for every $1 \leq j \leq n$ in the Chalmers-Metcalf operator $T$ for $Y$ and hence every functional $e_j^*$ belongs to some norming pair for $P$.

To prove that the condition (\ref{stala}) can not hold we shall consider certain specific projections $Q$. For a fixed real number $t>0$ let us consider a projection $Q: X \to Y$ given as $Q(z)=z-f(z)v$, where $v=w+tiu$ and $u \in Y$ is any vector with the real coordinates satisfying $\|u\|_{\infty}=1$. We clearly have $f(v)=1$ and moreover
$$\|Q(z)-P(z)\|_{\infty}=\|tf(z)u\|_{\infty}=t|f(z)| \|u\|_{\infty} \leq t \|z\|_{\infty}.$$
Because the equality holds for $z \in \ell_{\infty}^n$ such that $\|z\|_{\infty}=f(z)=1$ it follows that $\|Q-P\| = t$. To estimate the norm of $Q$, we shall use the equation (\ref{norma}). For $1 \leq j \leq n$ we have
$$(1-f_j)|v_j| + |1-f_jv_j| = (1-f_j)|w_j + tiu_j| + |1-f_jw_j - tf_j i u_j|$$
$$=(1-f_j)\sqrt{w_j^2 + t^2u^2_j} + \sqrt{(1-f_jw_j)^2 + t^2f_j^2u_j^2} \leq (1-f_j)\sqrt{w_j^2 + t^2} + \sqrt{(1-f_jw_j)^2 + t^2},$$
where the last inequality follows from the fact that $|f_j|, |u_j| \leq 1$ for every $1 \leq j \leq n$. Hence, by the condition \eqref{norma} we can estimate
$$\|Q\| \leq \max_{1 \leq j \leq n} \left ( (1-f_j)\sqrt{w_j^2 + t^2} + \sqrt{(1-f_jw_j)^2 + t^2} \right ).$$
Combining this with the condition (\ref{stala}) for an $\alpha$-strong uniqueness we get the inequality
$$\max_{1 \leq j \leq n} \left ( (1-f_j)\sqrt{w_j^2 + t^2} + \sqrt{(1-f_jw_j)^2 + t^2} \right )^{\alpha} \geq \|P\|^{\alpha} + rt^{\alpha}$$
for every $t > 0$. If we consider $t$ of the form $t=\frac{1}{k}$ for $k=1, 2, 3, \ldots$ then the maximum on the left side inequality has to be attained for a certain index $1 \leq j \leq n$ infinitely many times. To obtain a contradiction, it is therefore enough to prove that for any fixed index $1 \leq j \leq n$ on the left-hand side, this inequality fails for $t>0$ small enough.

So, let us fix $1 \leq j \leq n$. We recall that $\|P\| = (1-f_j)|w_j| + |1-f_jw_j|$. Hence, the inequality above can be rewritten as
$$\frac{\left ((1-f_j)\sqrt{w_j^2 + t^2} + \sqrt{(1-f_jw_j)^2 + t^2}\right )^{\alpha}  - \left ( (1-f_j)|w_j| + |1-f_jw_j| \right )^{\alpha}}{t^{\alpha}} \geq r.$$
However, calculating the limit of the left-hand side for $t \to 0^{+}$ with a help of the L'Hospital's rule we get
$$\lim_{t \to 0^+} \frac{\left ((1-f_j)\sqrt{w_j^2 + t^2} + \sqrt{(1-f_jw_j)^2 + t^2}\right )^{\alpha}  - \left ( (1-f_j)|w_j| + |1-f_jw_j| \right )^{\alpha}}{t^{\alpha}}$$
$$=\lim_{t \to 0^+}\left (  \frac{(1-f_j)t}{\sqrt{w^2_j+t^2}} + \frac{t}{\sqrt{(1-f_jw_j)^2 + t^2}} \right ) \cdot \frac{\left ((1-f_j)\sqrt{w_j^2 + t^2} + \sqrt{(1-f_jw_j)^2 + t^2}\right )^{\alpha-1}}{t^{\alpha-1}} $$
$$=\lim_{t \to 0^+} t^{2 - \alpha} \left (  \frac{1-f_j}{\sqrt{w^2_j+t^2}} + \frac{1}{\sqrt{(1-f_jw_j)^2 + t^2}} \right ) \left ((1-f_j)\sqrt{w_j^2 + t^2} + \sqrt{(1-f_jw_j)^2 + t^2}\right )^{\alpha-1}$$
$$=\lim_{t \to 0^+} t^{2 - \alpha} \left ( \frac{1-f_j}{|w_j|} + \frac{1}{|1-f_jw_j|} \right ) \left ((1-f_j)|w_j| + |1-f_jw_j|\right )^{\alpha-1}=0,$$
where in the last step we used the fact that $\alpha < 2$ (we note that $w_j>0, 1-f_jw_j>0$). We have obtained a contradiction and the proof is finished.

\end{proof}

Our next result is directly related to results established in Section \ref{sectionreal}. In Theorem \ref{twrpolytope} it was proved that in the complex polytope norms and their adjoints, a unique best approximation is automatically $2$-strongly unique, when real subspaces are considered and the ambient norm is determined by real vectors or functionals. In theorem below we establish a similar fact but for unique minimal projections. It should be noted that in this case Theorem \ref{twrpolytope} can not be applied directly as, contrary to the real case, we can not easily deduce that the dual space of linear operators $\mathcal{L}^*(X, X)$ is a complex polytope space or an adjoint complex polytope space, even when $X$ is such a space. This is because $X$ and $X^*$ can not be simultaneously complex polytope or adjoint complex polytope spaces, as shown in Theorem \ref{twrnorms}. However, as we shall see, Theorem \ref{twrpolytope} can still be applied, albeit in a less straightforward way.

\begin{twr}
\label{twrprojreal}
Let $X=(\mathbb{C}^n, \| \cdot \|)$ be either a complex polytope space with a real essential system of vertices or an adjoint complex polytope space with a real essential system of facets. Suppose that $Y \subseteq X$ is a real subspace, for which a minimal projection $P: X \to Y$ is unique. Then $P$ is a real projection (for every $v \in \mathbb{R}^n$ we have $P(v) \in \mathbb{R}^n$) and $P$ is a $2$-strongly unique minimal projection. 
\end{twr}

\begin{proof}
Throughout the proof we shall always understand a linear operator $L: \mathbb{C}^{n_1} \to \mathbb{C}^{n_2}$ to be \emph{real} if it transforms real vectors to real vectors, i.e. $L(\mathbb{R}^{n_1}) \subseteq \mathbb{R}^{n_2}$. We note that if $X$ has a real essential system of vertices $u_1, \ldots, u_N \in \mathbb{R}^n$, then the convex hull $\conv\{ \pm u_1, \ldots, \pm u_N \} \subseteq \mathbb{R}^n$ is a unit ball of a polytope norm in $\mathbb{R}^n$. In this case, the dual norm in $\mathbb{R}^n$ is also a polytope norm and we can assume that real functionals $\pm f_1, \ldots, \pm f_M$ are extreme points of the unit ball of the dual norm. Similarly, if $X$ has a real essential system of facets $f_1, \ldots, f_M$, then we define vectors $u_1, \ldots, u_N \in \mathbb{R}^n$ to be extreme points (along with their negatives) of the set $\{ x \in \mathbb{R}^n: |f_j(x)| \leq 1 \text{ for every } 1 \leq j \leq M\}$. The proof will be conducted mostly simultaneously in both situations.

Let us suppose that $\dim Y = n-d$ with $1 \leq d \leq n-1$ and $Y = \bigcap_{j=1}^{d} \ker g_j$, where $g_1, \ldots, g_d: \mathbb{C}^n \to \mathbb{C}$ are real functionals (since $Y$ is a real subspace such a choice is possible by Lemma \ref{realcond}). The projection $P$ can be written in the form
$$P(x) = x - \sum_{j=1}^{d} g_j(x) w_j,$$
for some vectors $w_1, \ldots, w_d \in \mathbb{C}^n$ such that $g_j(w_k) = \delta_{j, k}$ for any $1 \leq j, k \leq d$. ($\delta_{j, k}$ denotes the Kronecker delta). Let $\widetilde{w_j}=\re(w_j)$ for $1 \leq j \leq d$. Clearly, as $g_j$'s are real functionals we also have $g_j(\widetilde{w_k})=\delta_{j, k}$ for any $1 \leq j, k \leq d$ and a linear operator $\widetilde{P}:X \to Y$ given as
$$\widetilde{P}(x) = x - \sum_{j=1}^{d} g_j(x) \widetilde{w}_j$$
is a projection onto $Y$. Now, if $X$ has a real essential system of vertices $u_1, \ldots, u_N$ then $\|P\| = \max_{1 \leq k \leq N} \|P(u_k)\|$ and similarly $\|\widetilde{P}\| = \max_{1 \leq k \leq N} \|\widetilde{P}(u_k)\|$. However, by Lemma \ref{lemreal} for any fixed $1 \leq k \leq N$ we have
$$\|P(u_k)\| = \left \| u_k - \sum_{j=1}^{d} g_j(u_k) w_j \right \| = \left \| u_k - \sum_{j=1}^{d} g_j(u_k) \re(w_j) - i \sum_{k=1}^{d} g_j(u_k) \im(w_j) \right \|$$
$$\geq \left \| u_k - \sum_{j=1}^{d} g_j(u_k) \re(w_j) \right \| = \|\widetilde{P}(u_k)\|.$$
It follows that $\|\widetilde{P}\| \leq \|P\|$ and since by the assumption $P$ is a unique minimal projection, we must have $\widetilde{P}=P$. In particular, for every $1 \leq j \leq d$ we have $\widetilde{w_j}=w_j$, i.e. $w_j \in \mathbb{R}^n$ and $P$ is a real projection.

In the case of $X$ having a real essential system of facets $f_1, \ldots, f_M$ we use the fact that $\|P\| = \max_{1 \leq k \leq M} \|f_k \circ P\|_*$ and similarly $\|\widetilde{P}\| = \max_{1 \leq k \leq M} \|f_k \circ \widetilde{P}\|_*$, where $\| \cdot \|_*$ denotes the dual norm to $\| \cdot \|$ in $\mathbb{C}^n$. Then for any fixed $1 \leq k \leq M$ we have
$$ (f_k \circ P)(x) =  f_k(x) -  \sum_{j=1}^{d} g_j(x) f_k(w_j). $$
Since this is a real functional, reasoning like before but this time with Lemma \ref{lemreal} applied to the norm $\| \cdot \|_*$ (which has a real essential system of vertices) we conclude that $\|f_k \circ P \|_* \geq \|f_k \circ \widetilde{P}\|$ for every $1 \leq k \leq M$ and hence $\|P\| \geq \|\widetilde{P}\|$. Therefore, again we conclude that for every $1 \leq j \leq d$ we have $w_j \in \mathbb{R}^n$ and, in consequence, $P$ is a real projection.

The fact that $P$ is a unique minimal projection from $X$ onto $Y$ means that $0$ is a unique best approximation for $P$ in the subspace $\mathcal{L}_Y(X, Y)$. Let a subspace $\mathcal{V} \subseteq \mathcal{L}(X, Y)$ be defined as $\mathcal{V} = \lin \{ \mathcal{L}_Y(X, Y) \cup P \}$. For an operator $L \in \mathcal{V}$ we define $\|L\|_0$ as
$$\|L\|_0 = \max_{1 \leq j \leq N, 1 \leq k \leq M} |f_k(L(u_j))|,$$
regardless of whether $X$ has a real essential system of vertices $u_1, \ldots, u_N$ or a real essential system of facets $f_1, \ldots, f_M$ (as explained in the beginning of the proof). Clearly, $\| \cdot \|_0$ is a semi-norm on $\mathcal{V}$ and $\|L\|_0 \leq \|L\|$ for every $L \in \mathcal{V}$, as $u_1, \ldots, u_N \in S_{X}$ and $f_1, \ldots, f_M \in S_{X^*}$. We shall establish the following three claims.

\begin{enumerate}
\item If $L \in \mathcal{V}$ is a real operator, then $\|L\|_0 = \|L\|.$
\item If $L \in \mathcal{V}$ and $\|L\|_0=0$, then $L \equiv 0$.
\item If $L \in \mathcal{L}_Y(X, Y)$ and $\|P-L\|_0 \leq \|P\|_0$, then $L \equiv 0$.
\end{enumerate}

First we will demonstrate how to complete the proof with the above three claims already established. Second claim implies immediately that $\| \cdot \|_0$ is a norm on $\mathcal{V}$. We note that $\dim \mathcal{L}_Y(X, Y) = d(n-d)$ with a basis given by rank one operators $g_k \otimes y_j$, where $1 \leq j \leq n-d$, $1 \leq k \leq d$ and $y_1, \ldots, y_{n-d}$ is any basis of $Y$. We can further assume that $y_1, \ldots, y_{n-d} \in \mathbb{R}^n$, as $Y$ is a real subspace. Thus $\dim \mathcal{V} = d(n-d)+1$ with a basis consisting of the aforementioned operators and the projection $P$. Hence, if a linear operator $L \in \mathcal{V}$ is of the form $L=aP + \sum_{k, j} a_{k, j} g_k \otimes y_j$, then 
$$f_s(L(u_t))=af_s(P(u_t)) + \sum _{k, j} a_{k, j} g_k(u_t) f_s(y_j).$$
Since $u_j$'s and $y_j$'s are all real vectors and $P, g_j, f_j$ are all real operators, when $\mathcal{V}$ is identified with $\mathbb{C}^{d(n-d)+1}$, the functional $L \to f_s(L(u_t))$ is also real (if $a, a_{k, j} \in \mathbb{R}$, then $f_s(L(u_t)) \in \mathbb{R}$). It follows that, when $\mathcal{V}$ is identified with $\mathbb{C}^{d(n-d)+1}$, the norm $\| \cdot \|_0$ is an adjoint complex polytope norm with a real essential system of facets, since it is a maximum over a finite number of real functionals. The third claim implies now that $0$ is a unique best approximation for $P$ in $\mathcal{L}_Y(X, Y)$ when the norm $\| \cdot \|_0$ is considered and hence it follows from Theorem \ref{twrpolytope} that $0$ is a $2$-strongly unique best approximation for $P$ in $\| \cdot \|_0$. Hence, there exists a constant $r>0$ such that for any $L \in \mathcal{L}_Y(X, Y)$ we have
$$\|P-L\|_0^2 \geq \|P\|_0^2 + r \|L\|_0^2.$$
Moreover, $\|P\|_0 = \|P\|$ from the first claim and there exists a constant $c>0$ such that $\|L\|_0 \geq c \|L\|$, as these two norms are equivalent on a finite-dimensional subspace $\mathcal{V}$ (with the obvious estimate in the other direction being $\|L\|_0 \leq \|L\|)$. Altogether we have
$$\|P-L\|^2 \geq \|P-L\|_0^2 \geq \|P\|_0^2 + r \|L\|_0^2 \geq \|P\|^2 + cr \|L\|^2.$$
Now, if $Q$ is any linear projection from $X$ onto $Y$, then using the inequality above for $L = P-Q \in \mathcal{L}_Y(X, Y)$ we get
$$\|Q\|^2 \geq \|P\|^2 + cr \|P-Q\|^2,$$
which proves that $P$ is a $2$-strongly unique minimal projection. Thus, the desired conclusion will follow when the three claims above are established.

Let us start with the first claim. If $X$ has a real essential system of vertices $u_1, \ldots, u_N$, then for every linear operator $L \in \mathcal{V}$ we have $\|L\| = \max_{1 \leq j \leq N} \|L(u_j)\|$. Therefore, if $L$ is a real operator then the vectors $L(u_j)$ are also real and for every real vector $v \in \mathbb{R}^n$ we clearly have $\|v\| = \max_{1 \leq j \leq M} |f_j(v)|$ (as $f_1, \ldots, f_M \in \mathbb{R}^n$ are defined as extreme points of the dual norm in this case). Thus $\|L\|=\|L\|_0$ in this situation. Similarly, if $X$ has a real essential system of facets $f_1, \ldots, f_M$, then $\|L\| = \max_{1 \leq j \leq M} \|f_j \circ L\|_*$. If $L$ is a real operator then the functionals $f_j \circ L$ are also real and for every real functional $f: \mathbb{C}^n \to \mathbb{C}$ we have $\|f\|_* = \max_{1 \leq j \leq N} |f(u_j)|$. Again it follows that $\|L\|_0=\|L\|$.

For the second claim, let us suppose that $L \in \mathcal{V}$ is such that $\|L\|_0=0$ and let us write $L$ in the form $L=aP + \sum_{k, j} a_{k, j} g_k \otimes y_j$ (where again $y_1, \ldots, y_{n-d}$ is a fixed real basis of $Y$). Let us also write $a= b + ci$ and $a_{k, j} = b_{k, j} + c_{k, j}i$. Then for $1 \leq s \leq M, 1 \leq t \leq N$ we have
$$0 = f_s(L(u_t))=af_s(P(u_t)) + \sum _{k, j} a_{k, j} g_k(u_t) f_s(y_j)$$
$$=\left (bf_s(P(u_t)) + \sum _{k, j} b_{k, j} g_k(u_t) f_s(y_j) \right ) + i \left ( cf_s(P(u_t)) + \sum _{k, j} c_{k, j} g_k(u_t) f_s(y_j) \right ) .$$
By comparing real and imaginary parts it follows that
$$bf_s(P(u_t)) + \sum _{k, j} b_{k, j} g_k(u_t) f_s(y_j)  = cf_s(P(u_t)) + \sum _{k, j} c_{k, j} g_k(u_t) f_s(y_j) = 0.$$
Thus, for real operators $L_1, L_2 \in \mathcal{V}$ defined as $L_1 = bP + \sum_{k, j} b_{k, j} g_k \otimes y_j $ and $L_2=cP + \sum_{k, j} c_{k, j} g_k \otimes y_j$ we have $\|L_1\|_0=\|L_2\|_0=0$. However, since $L_1, L_2$ are real operators it follows that $\|L_1\|=\|L_2\|=0$ by the first claim. Hence $L_1, L_2 \equiv 0$ and also $L \equiv 0$. This proves the second claim.

For the third claim, let us take a linear operator $L \in \mathcal{L}_Y(X, Y)$ such that $\|P-L\|_0 \leq \|P\|_0=\|P\|$ and again let us write it in the form $L=\sum_{k, j} a_{k, j} g_k \otimes y_j$ with $a_{k, j} = b_{k, j} + i c_{k, j}$. Arguing as in the previous step we see that $\|P-L_1\|_0 \leq \|P-L\|_0$, where $L_1 = \sum_{k, j} b_{k, j} g_k \otimes y_j$. However, $L_1$ is a real operator and thus $\|P-L_1\|=\|P-L_1\|_0$. Hence $\|P-L_1\| \leq \|P\|$ and since $P$ is a unique minimal projection onto $Y$, it follows that $L_1 \equiv 0$. Therefore $\|P-L\|_0=\|P-iL_2\|_0$. However, for any $1 \leq s \leq M, 1 \leq t \leq N$ we have
$$|f_s((P-iL_2)(u_t))| = \left | f_s(P(u_t)) - i \sum_{k, j} c_{k, j} g_k(u_t) f_s(y_j) \right | \geq |f_s(P(u_t))|$$
with the equality if and only if $\sum_{k, j} c_{k, j} g_k(u_t) f_s(y_j)=0$. Since this has to hold for all $1 \leq s \leq M, 1 \leq t \leq N$ it follows that $\|L_2\|_0=0$ and therefore $L_2 \equiv 0$. This proves that $0$ is a unique best approximation for $P$ in $\mathcal{L}_Y(X, Y)$ when the norm $\| \cdot \|_0$ is considered and the proof is finished.

\end{proof}

We note that in the previous result some assumption about the ambient norm is necessary, as the Example \ref{example1dim} shows. A $1$-dimensional subspace given in this example is a real subspace of $\mathbb{C}^n$, yet a unique minimal projection is not $2$-strongly unique in this case.

It is easy to observe that Theorem \ref{twrprojreal} gives an alternative proof of the $2$-strong uniqueness of minimal projections onto hyperplanes of $\ell_{\infty}^n$, as in Theorem \ref{twrlinfty}. However, we have preferred to give a proof based on the injectivity of the Chalmers-Metcalf operator, as it was more related to the latter part of the theorem, where we have shown that the $2$-strong uniqueness can not be improved. In the last result of this section, we shortly demonstrate how to use Theorem \ref{twrprojreal} in the case of hyperplane projections of the space $\ell_1^n$. Although the conditions for uniqueness of minimal projections onto hyperplanes of $\ell_1^n$ are known, we do not state them here as they are quite complicated (see Section II.3 in \cite{lewickihabilitacja}).

\begin{twr}
\label{twrl1proj}
Let $Y \subseteq \ell_{1}^n$ be a subspace given as $Y=\ker f$ for some non-zero vector $f=(f_1, \ldots, f_n) \in \mathbb{C}^n$. If a minimal projection from $\ell_{1}^n$ onto $Y$ is unique, then it is $2$-strongly unique.
\end{twr}

\begin{proof}
Since a mapping of the form
$$\ell_{1}^n \ni (x_1, x_2, \ldots, x_n) \to (a_1x_1, a_2x_2, \ldots, a_nx_n) \in \ell_{1}^n$$
is an isometry for any complex scalars $a_i$ of modulus $1$, we can suppose that $f_j$ is a real number for any $1 \leq j \leq n$. In this case, $Y$ is a real subspace of a complex polytope space $\ell_1^n$ and the desired conclusion follows immediately from Theorem \ref{twrprojreal}.
\end{proof}

We note that the same argument yields that a uniqueness of minimal projection onto a $1$-dimensional subspace of $\ell_{\infty}^n$ or $\ell_1^n$ is equivalent to a $2$-strong uniqueness. We do not know however, what happens in the case of subspaces of dimension larger than $1$ and smaller than $n-1$.

\end{document}